\documentclass[11pt, letterpaper]{amsart}
\pdfoutput=1
\usepackage{amssymb,latexsym, xcolor
}
\usepackage{amsthm,amssymb}
\usepackage[dvips, hcentering,
includehead,width=14.2cm,
top=0.5cm,
height=21.6cm,
paperheight=23.8cm,paperwidth=17cm
]{geometry}

\usepackage{hyperref}
\usepackage{enumitem}

\usepackage{tikz-cd}
\usetikzlibrary{decorations.markings} 

\usepackage{mathtools}

\usepackage{ccaption}
\changecaptionwidth
\captionwidth{5.6in}

\usepackage{bm}

\setitemize{leftmargin=.2in}
\setenumerate{leftmargin=.3in}

\newcommand{\cS}[1]{{\noindent\textsf{\color{purple}$\blacksquare$~#1~$\blacksquare$}}}
\newcommand{\hide}[1]{}

\newtheorem{theorem}{Theorem}[section]
\newtheorem{proposition}[theorem]{Proposition}

\newtheorem{corollary}[theorem]{Corollary}
\newtheorem{lemma}[theorem]{Lemma}

\theoremstyle{definition}

\newtheorem{remark}[theorem]{Remark}

\newtheorem{example}[theorem]{Example}
\newtheorem{definition}[theorem]{Definition}
\newtheorem{problem}[theorem]{Problem}

\numberwithin{equation}{section}

\def\ZZ{{\mathbb Z}}

\def\wind{{\operatorname{wind}}}

\def\GL{{\operatorname{GL}}}
\def\In{{\operatorname{In}}}
\def\Out{{\operatorname{Out}}}

\newcommand\sgn{\operatorname{sgn}}

\def\tr{{\operatorname{tr}}}

\newcommand{\boxQ}{\overset{\hspace*{0.4pt}\scalebox{0.4}{$\bm{\Box}$}}{Q}} 

\newcommand{\Qinit}{Q_0}
\newcommand{\Qtmut}{Q_t}
\newcommand{\Lgraph}{L}

\title{Mutation-acyclic quivers are totally proper}

\begin{document}

\author{Scott Neville}
\address{\hspace{-.3in} Department of Mathematics, University of Michigan,
Ann Arbor, MI 48109, USA}
\email{nevilles@umich.edu}

\date{\today}

\thanks{Partially supported by NSF grants DMS-1840234, DMS-2054231, DMS-2348501.
}



\begin{abstract}
Totally proper quivers, introduced by S.~Fomin and the author \cite{COQ}, have many useful properties including powerful mutation invariants.
We show that every mutation-acyclic quiver (i.e., a quiver that is mutation equivalent to an acyclic one) is totally proper.
This yields new necessary conditions for a quiver to be mutation-acyclic.
In particular, we show that a generalization of the Markov invariant for $3$-vertex quivers applies to all mutation-acyclic quivers.
Only finitely many acyclic quivers share the same Markov invariant.
%
%
%
\end{abstract}

\maketitle



\thispagestyle{empty}

\section{Introduction}
\emph{Quivers} are finite directed graphs without oriented cycles of length $1$ or $2$. 
\emph{Mutations} are operations that transform a quiver, based on a choice of a vertex.
These notions are foundational in the study of cluster algebras~\cite{CA1}.
A \emph{mutation invariant} is a characteristic of a quiver that is preserved under mutations.
Mutation invariants are helpful for deciding whether two quivers are mutation equivalent or not,
i.e., whether there is a sequence of mutations that transforms one quiver into the other.
One example is the \emph{Markov invariant} of a $3$-vertex quiver, which essentially characterizes the $3$-vertex quivers which are \emph{mutation-acyclic} (i.e., those quivers which are  mutation equivalent to an acyclic quiver) \cite{BBH}.
See \cite{casals-binary, ervinPrefork, CAtextbook1-3, Seven3x3, Seven-congruence} for additional examples of known mutation invariants. 

A \emph{cyclically ordered quiver} (COQ) is a pair $(Q, \sigma)$ where $Q$ is a quiver and~$\sigma$ a cyclic ordering of its vertices.
Cyclically ordered quivers were introduced in~\cite{COQ} to develop new powerful mutation invariants. 
A~mutation in a COQ $(Q, \sigma)$ transforms~$Q$ by the usual mutation rule,
while simultaneously changing the cyclic ordering~$\sigma$ in a prescribed way.
We note that mutations of COQs are only allowed at the vertices that satisfy a certain \emph{properness} condition. 
It is this condition that ultimately enables the introduction of new mutation invariants. 


In this paper, we prove that for one important class of quivers, the properness requirement can be lifted, 
so that mutations at all vertices are allowed and all invariants developed in~\cite{COQ} become true mutation invariants.
Specifically, we show that in any mutation-acyclic quiver~$Q$, 
every vertex~$v$ is proper for a particular canonical cyclic ordering~$\sigma$ on~$Q$. 
After a mutation at~$v$, we obtain a new COQ $(Q',\sigma')$, with canonical cyclic ordering~$\sigma'$,
that again has the same property: all the vertices are proper, so we can mutate at any one of them, 
and the process continues. 

This theorem yields new mutation invariants of mutation-acyclic quivers, 
as well as new tools for proving that various quivers are not mutation-acyclic. 
One such quiver is shown in Figure~\ref{fig:Extended8}.
We give a short proof that this quiver is not mutation-acyclic in Example~\ref{eg:extended not acyclic}.


\begin{figure}[ht]
\begin{center}
\includegraphics[alt={A graph with ten vertices and 18 arrows, arranged into an equalateral triangular grid. There are four vertices on each side, with one interior vertex, thus there are nine equalateral triangular faces (each supported by three vertices) which cover the whole. Each face has the arrows oriented cyclically, directed counter-clockwise.}]{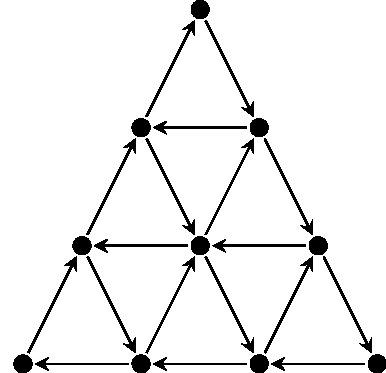}
\vspace{3pt}
\end{center}
\caption{An (unlabeled) triangular grid quiver with $4$ vertices on each side. 
Readers may recognize it as the default quiver in B.~Keller's mutation applet \cite{KellerApp}.}
\vspace{-15pt}

\label{fig:Extended8}
\end{figure}

As mentioned above, mutation at a vertex~$v$ in a COQ $(Q,\sigma)$
is only allowed if $v$ satisfies a combinatorial condition of ``properness.'' 
Informally, $v$~is proper if every 2-arrow oriented path through~$v$
travels ``clockwise,'' i.e., in the direction of the cyclic ordering~$\sigma$. 
A COQ is \emph{proper} if every vertex~$v$ in it is proper 
(possibly after applying a sequence of transpositions called \emph{wiggles}). 
Thus, we can mutate at any vertex in a proper COQ $(Q,\sigma)$ 
and get a new COQ $(Q',\sigma')$---but $(Q',\sigma')$ will not necessarily be proper.
If it is the case that applying any sequence of mutations to $(Q,\sigma)$ results in a proper COQ, 
then we say that $(Q,\sigma)$ is \emph{totally proper}.

We have shown in  \cite[Theorem~14.3]{COQ} that 
a quiver $Q$ can be upgraded to a totally proper COQ $(Q,\sigma)$ 
in at most one way (up to wiggles).
Moreover, the candidate cyclic ordering $\sigma=\sigma_Q$ can be constructed efficiently \cite[Algorithm~14.12]{COQ}.
We say that a quiver $Q$ is totally proper if the COQ $(Q, \sigma_Q)$ is totally proper.

As shown in \cite[Corollary 14.2]{COQ}, 
totally proper quivers have a powerful mutation invariant, which we recall next.
This invariant is constructed as follows. 

\smallskip

\noindent
\underline{Input}: a totally proper quiver $Q$. \\
\underline{Step~1}. Construct the canonical cyclic ordering $\sigma=\sigma_Q$. \\
\underline{Step~2}. ``Tear'' $\sigma$ into a linear ordering $<\,$. \\
\underline{Step~3}. Construct the skew-symmetric exchange matrix~$B=B_Q$, with rows and columns ordered according to~$<$. \\
\underline{Step~4}. Construct the \emph{unipotent companion}~$U$, the unique 
unipotent upper-triangular matrix such that $B=U^T - U$. \\
\underline{Output}: the \emph{integral congruence class} of $U$, i.e., 
the set $\{G U G^T \mid G \in \GL_n(\ZZ)\}$.

\smallskip

We have shown in \cite{COQ} that wiggles, cyclic shifts of the linear ordering, 
and (proper) mutations of COQs preserve the integral congruence class of the unipotent companion. \linebreak[3]
While the resulting invariant of proper mutations is very powerful,
it is not easy to use in practice, since 
the problem of deciding whether two upper-triangular matrices 
are congruent over $\GL_n(\ZZ)$ seems to be difficult.  
In~\cite{COQ}, we bypassed this difficulty as follows. 
It is well known (and easy to see) that the integral congruence class of a matrix $U\in \GL_n(\ZZ)$
uniquely determines the conjugacy class of its \emph{cosquare}~$U^{-T} U$.
It follows that whenever two COQs are related by proper mutations,
the cosquares of their respective unipotent companions must be conjugate in~$\GL_n(\ZZ)$. 
This conjugacy condition can be verified efficiently,
though the algorithm is nontrivial \cite{BHJ-opensource, EHO-magma}.

The cosquare $U^{-T} U$ can be used to construct other invariants of proper mutations, 
such as the monic characteristic polynomial of $U^{-T} U$, 
which we call the \emph{Alexander polynomial} of the COQ~$Q$.
(While much simpler to compute, the Alexander polynomial is less powerful
than the conjugacy class of $U^{-T} U$:
integer matrices may have the same characteristic polynomial while 
not being conjugate in~$\GL_n(\ZZ)$.) 

\pagebreak[2]

The main result of this paper is the following theorem \cite[Conjecture~15.9]{COQ}. 
\begin{theorem}
\label{thm:acyclic totally proper}
Every mutation-acyclic quiver is totally proper. 
\end{theorem}

\begin{remark}
\label{rem:the order}
It suffices to show that every acyclic quiver is totally proper.
Further\-more, as noted in {\cite[Observation~11.6]{COQ}}, any acyclic quiver has a proper cyclic ordering 
obtained by ``closing up'' any linear ordering compatible with the orientations of the quivers' arrows.
(All such cyclic orderings are related by wiggles.)
Our proof of Theorem~\ref{thm:acyclic totally proper} shows that these cyclic orderings are totally proper.
\end{remark}

In the course of proving Theorem~\ref{thm:acyclic totally proper},
we establish a few results that may be of independent interest.
In particular, Theorem~\ref{thm:no vortex} gives a simple combinatorial constraint on the orientations of arrows in mutation-acyclic quivers.
Both Theorem~\ref{thm:no vortex} and the ensuing Lemma~\ref{lem:Q tilde acyclic} assert that certain subquivers of a mutation-acyclic quiver must be acyclic. 
Should someone wish to compute with COQs and unipotent companions, Lemmas~\ref{lem:mutates like B v1gen} and \ref{lem:mutates like B v2gen} generalize and remove conditions from {\cite[Lemma 7.2 and Remark 7.3]{COQ}}.
This can speed up and simplify the computations. 

An alternative proof of Theorem~\ref{thm:acyclic totally proper} was independently discovered by Hugh Thomas, using categorification \cite{HughPC}.

One consequence of Theorem~\ref{thm:acyclic totally proper} is the existence of a \emph{Markov invariant} for all mutation acyclic quivers. 
We establish some basic inequalities for its coefficients and give a simple combinatorial formulation in the acyclic case (Proposition~\ref{prop:markov positive}), including that the Markov invariant is positive (respectively, at least $n-1$) for any (respectively, any connected) mutation-acyclic quiver (Corollary~\ref{cor:acyclic positive}). 
As in the $3$-vertex case, there are only finitely many acyclic quivers with a given Markov invariant. 
Thus there are only finitely many acyclic quivers with a given Alexander polynomial (Corollary~\ref{cor:markov finite}). 

Amanda Schwartz~\cite{SchwartzTree} studied the Alexander (and HOMFLY) polynomials of the link associated to certain plabic graphs, which in turn correspond to \emph{tree} quivers (those whose underlying undirected graph is a tree).
Our constructions agree on the Alexander polynomial.
She gave a combinatorial description of all of the coefficients of the Alexander polynomial in this case~\cite[Theorem 4.2]{SchwartzTree}.

A \emph{quasi-Cartan companion} is a symmetric matrix associated to a quiver, in analogy to the Cartan companion of a dynkin diagram.
The existence of a \emph{quasi-Cartan companion} with various properties has been used to characterize quivers of finite type \cite{BGZsymmetrizable, SevenSemiPositive}. 
Our proof of Theorem~\ref{thm:acyclic totally proper} uses the results of A.~Seven \cite{SevenAMatrices} to identify the correct cyclic ordering to associate with a given mutation-acyclic quiver, and to guarantee that certain subquivers of a mutation-acyclic quiver are acyclic (Theorem~\ref{thm:no vortex}).
Some of our technical lemmas also extend to quasi-Cartan companions, see Corollary~\ref{cor:A mutates like B}.

\subsection*{Overview of the paper}
Section~\ref{sec:quivs} establishes notation and reviews some material that does not involve COQs.
Much of this section is devoted to restating the results of A.~Seven \cite{SevenAMatrices}. 
Section~\ref{sec:coqs} is a condensed summary of the necessary results and notation from \cite{COQ}. 
It contains new examples, and Proposition~\ref{prop:markov positive} gives an alternative description of the Markov invariant for acyclic quivers.

In Section~\ref{sec:tech} we establish a handful of technical lemmas.
Lemmas~\ref{lem:mutates like B v1gen} and \ref{lem:mutates like B v2gen} make it easier to compute the unipotent companion of a COQ after a proper mutation.
Theorem~\ref{thm:no vortex} generalizes \cite[Corollary~13.3]{COQ} using the results of A.~Seven \cite{SevenAMatrices}.
Section~\ref{sec:thm} contains the proof of Theorem~\ref{thm:acyclic totally proper}.
We begin by defining $U=(A-B)/2$, where $A$ comes from Seven's construction and $B$ is the usual exchange matrix. 
The fact that the integral congruence class of $U$ is a mutation invariant follows quickly from Proposition~\ref{prop:A matrix mutation acyclic}.
The bulk of the effort is then spent showing that $U$ is indeed a unipotent companion. 

In Section~\ref{sec:cor}, we give some applications and corollaries of Theorem~\ref{thm:acyclic totally proper}. 
This includes a description of invariants for some particular families of quivers, inequalities satisfied by the coefficients of the Alexander polynomial for mutation-acyclic quivers (Corollaries~\ref{cor:acyclic positive},~\ref{cor:markov finite}, and~\ref{cor:acyclic det inequality}), and several examples of quivers that are shown to be not mutation-acyclic. 
This section also contains Examples~\ref{eg:not unimodal} and~\ref{eg:bipartite acyclic}, which respectively show that not all Alexander polynomials of acyclic quivers are unimodal, but that many are.

\subsection*{Acknowledgements}
I would like to thank my advisor, Sergey Fomin, for his guidance, comments, and suggestions; 
Grayson Moore, for reading and commenting on an early draft; 
and Danielle Ensign for her software assistance.
I am also grateful to Roger Casals, Tucker Ervin, Hugh Thomas, and Sara Billey for stimulating discussions.

I have used \textsc{Magma} and \texttt{Sage} for various computations.
This research was supported in part through computational resources and services provided by Advanced Research Computing at the University of Michigan, Ann Arbor. 

\newpage

\section{Quivers, mutation, and quasi-Cartan companions}
\label{sec:quivs}

We begin by reviewing notation, definitions and results from the literature which do not involve cyclically ordered quivers.
This includes a summary of much from \cite{SevenAMatrices}. 


\begin{definition}
A \emph{quiver} is a finite directed graph with parallel edges allowed, but no directed $1$ or $2$-cycles.
Directed edges in a quiver are called \emph{arrows}.
Each vertex is marked as \emph{mutable} or \emph{frozen}. 
Unless otherwise indicated, all vertices are mutable.

We use the notation $ u \rightarrow v$ to assert that there is at least one arrow from $u$ to~$v$.
We use the notation $u \stackrel{x}{\rightarrow} v$ to denote that there are $x$ arrows from vertex $u$ to vertex~$v$.
Let $\In(v) = \{u | u \rightarrow v\}$ and $\Out(v) = \{u | u \leftarrow v\}$ respectively denote the inset and outset of $v$ in a quiver.
\end{definition}

\begin{remark}
By default, our quivers have labeled vertices. 
Thus, we distinguish between isomorphic quivers on the same set of vertices that differ from each other by a permutation of this set. 
\end{remark}

\begin{definition}
To \emph{mutate} a quiver $Q$ at a mutable vertex $v$, do the following:
\begin{enumerate}
\item for each oriented path $u \rightarrow v \rightarrow w$, draw a new arrow $u \rightarrow w$ (thus, if there are $x$ arrows from $u \rightarrow v$ and $y$ arrows $v \rightarrow w$, we add $xy$ arrows $u \rightarrow w$);
\item reverse all arrows adjacent to $v$;
\item remove oriented $2$-cycles (one cycle at a time) until we again have a quiver. 
\end{enumerate}
We denote the resulting quiver by $\mu_v(Q)$.
Mutation is an involution, $\mu_v ( \mu_v(Q)) = Q$.
Mutation does not change which vertices are mutable or frozen.
\end{definition}

\begin{definition}
Two quivers are \emph{mutation equivalent} if they are related by a sequence of mutations at mutable vertices.
The \emph{mutation class} $[Q]$ is the set of quivers mutation equivalent to $Q$.
\end{definition}

\begin{definition}
A quiver is \emph{acyclic} if it does not contain any directed cycles.
A quiver is \emph{mutation-acyclic} if it is mutation equivalent to an acyclic quiver.
\end{definition}

\begin{example}
Only one of the two quivers in the Figure~\ref{fig:acyclic and not} is mutation-acyclic.
Can you guess which is which?  
\end{example}

\begin{figure}[ht]
{
\newcommand\cyclicWeights[5]{
	\filldraw[black] (#4,#5)++(150:0.8cm) circle (2pt) node[above left=-1pt] {$v_1$} coordinate (a);
	\filldraw[black] (#4,#5)++(30:0.8cm) circle (2pt) node[above right=-1pt] {$v_2$} coordinate (b);
	\filldraw[black] (#4,#5)++(-90:0.8cm) circle (2pt) node[below right=-1pt] {$v_3$} coordinate (c);
	\draw[black, -{stealth}, shorten >=3pt, shorten <= 3pt] (a) -- (b) node [midway, above] {#1};
	\draw[black, -{stealth}, shorten >=3pt, shorten <= 3pt] (b) -- (c) node [midway, right] {#2};
	\draw[black, -{stealth}, shorten >=3pt, shorten <= 3pt] (c) -- (a) node [midway, left] {#3};
}

\includegraphics[alt={Two cyclically oriented 3-vertex quivers, both with vertices labeled v1, v2 and v3. The cyclically oriented quiver on the left has multiplicities (weights) 5, 12, and 57 (between vertices 1 and 2, 2 and 3, and 3 and 1 respectively). The cyclically oriented quiver on the right has multiplicities 5, 12, and 58.}]{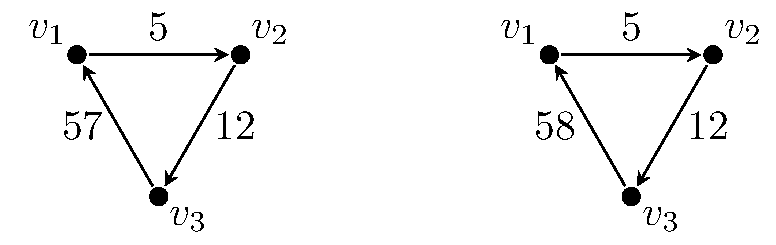}
} 

\caption{Two quivers on $3$-vertices. The quiver on the left \emph{is not} mutation-acyclic, while the quiver on the right \emph{is} mutation-acyclic (by mutating at $v_2$ and then~$v_1$).}
\label{fig:acyclic and not}
\end{figure}

\begin{definition}
\label{def:B matrix}
For a quiver $Q$, the \emph{$B$-matrix} $B_Q = (b_{uv})$ (also known as the \emph{exchange matrix}) 
is the skew-symmetric adjacency matrix of $Q$. 
By convention we have $b_{uv} > 0$ when there are $b_{uv}$ arrows $u \rightarrow v$, 
and $b_{uv} < 0$ when there are $-b_{uv}$  arrows $u \leftarrow v$.
\end{definition}

\begin{definition}
\label{def:subquiv}
A \emph{(full) subquiver} is an induced subgraph of a quiver.
We call the remaining vertices of a subquiver its \emph{support}.
The \emph{mutable part} 
of a quiver $Q$ is the subquiver supported by its mutable vertices.
\end{definition}

\begin{definition}
The \emph{principal extension} (also known as \emph{framing}) 
of a quiver $Q$ is a quiver $\widehat Q$ formed by adding a new frozen vertex $v'$ and a single arrow $v' \stackrel{}{\rightarrow} v$ for each vertex~$v$.
Thus $Q$ is the mutable part of $\widehat Q$ and $\widehat Q$ has twice as many vertices as~$Q$.
\end{definition}

\begin{figure}[ht]
{
\newcommand\extht{-1.5cm}
\newcommand\boxSize{2.5pt}
\includegraphics[alt={The principle framing of an acyclic four vertex quiver, with vertices v1, v2, v3, and v4. There is one arrow from v1 to v2, one arrow from v1 to v3, three arrows from v1 to v4, two arrows from v2 to v4, and one arrow from v3 to v4. The four frozen vertices are denoted v1', v2', v3', v4', and each come with a single arrow from v i' to v i.}]{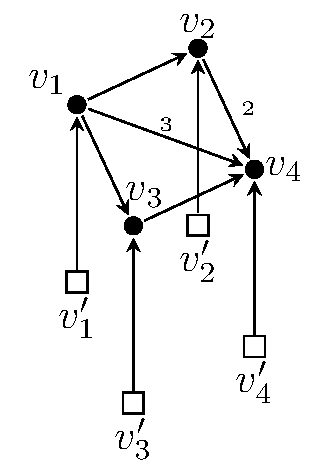}
} 

\caption{The black circle vertices $v_i$ support an acyclic quiver $Q$. 
With the additional square (frozen) vertices $v_i'$, we get the principal framing $\widehat Q$.}
\label{fig:principal frame}
\end{figure}

We will use the following notation for mutation classes and associated matrices and data. 
(Cf.\ also Definitions~\ref{def:A t} and~\ref{def:U t}.)

\begin{definition}
\label{def:indexing}
Let $\mathbb{T}_n$ be an $n$-regular tree whose edges are labeled by the integers $1, \ldots, n$, so that each edge label appears next to each vertex.
We write $t \stackrel{i}{\text{---}} t'$ to indicate that an edge labeled $i$ joins the vertices $t$ and $t'$ in $\mathbb{T}_n$.
The tree $\mathbb{T}_n$ always comes with a distinguished vertex~$t_0$. 

Fix a quiver $Q_0$ with $n$ mutable linearly ordered vertices $v_1 <  \cdots < v_n$ (and no frozen vertices).
We index the mutation class~$[\widehat Q_0]$ by simultaneously assigning a quiver $\tilde Q_t \in [\widehat Q_0]$ 
to each vertex $t$ in $\mathbb{T}_n$, so that $\tilde Q_{t_0}=\widehat Q_0$
and $\tilde Q_t = \mu_{v_i}(\tilde Q_{t'})$ whenever $t \stackrel{i}{\text{---}} t'$.
We call $Q_0$ the \emph{initial quiver}. 
For each vertex $t \in \mathbb T_n$, let~$Q_t$ be the mutable part of~$\tilde Q_t$.
Let $B_t = (b_{v_i v_j; t})$ be the $n \times n$ exchange matrix of~$Q_t$.
Similarly, let $C_t = (c_{v_i' v_j;t})$ be $n \times n$ bipartite adjacency matrix between frozen vertices and mutable vertices. 
Thus
$$b_{v_iv_j;t} = \#\{ \text{arrows } v_i \rightarrow v_j \text{ in } Q_t\} - \#\{ \text{arrows } v_i \leftarrow v_j  \text{ in } Q_t\},$$
$$c_{v_i'v_j;t} = \#\{ \text{arrows } v_i'\rightarrow v_j  \text{ in } \tilde Q_t\} - \#\{ \text{arrows } v_i' \leftarrow v_j  \text{ in } \tilde Q_t\}.$$
The matrix $C_t$ is called the \emph{$C$-matrix}, and its columns $\bold c_{v_i;t}$ are called \emph{$\bold c$-vectors}~\cite{Nakanishi-Zelevinsky}.
Note that $B_t$ is skew-symmetric, while $C_t$ need not have any symmetries.

Whenever we use a quiver or matrix indexed by $t$ we will specify an initial quiver, or relevant conditions on that choice.
We will often take an acyclic initial quiver $Q_0$. 
\end{definition}

\begin{example}
Let $Q_0$ be the acyclic quiver with vertices $v_1,v_2,v_3,v_4$, and arrows $v_1 \stackrel{}{\rightarrow} v_2 \stackrel{2}{\rightarrow} v_4$, $v_1\stackrel{}{\rightarrow} v_3 \stackrel{}{\rightarrow} v_4$, and $v_1 \stackrel{3}{\rightarrow} v_4$ (see Figure~\ref{fig:principal frame} for $\widehat Q_0$).
With initial quiver~$Q_0$ (and linear order $v_1 < v_2 < v_3 <v_4$), we have
$$ B_{t_0} = \begin{pmatrix} 0 & 1 & 1 & 3 \\ -1 & 0 & 0 & 2 \\ -1 & 0 & 0 & 1 \\ -3 & - 2 & - 1 & 0\end{pmatrix}, \quad \quad C_{t_0} = \begin{pmatrix} 1 & 0 & 0 & 0 \\ 0 & 1 & 0 & 0 \\ 0 & 0 & 1 & 0 \\ 0 & 0 & 0 & 1\end{pmatrix}.$$
Indeed, by construction $C_{t_0}$ is the identity matrix regardless of the initial quiver.

If $t \stackrel{2}{\text{---}} t_0$ then
$$ B_t = \begin{pmatrix} 0 & -1 & 1 & 5 \\ 1 & 0 & 0 & -2 \\ -1 & 0 & 0 & 1 \\ -5 & 2 & - 1 & 0\end{pmatrix}, \quad \quad C_t = \begin{pmatrix} 1 & 0 & 0 & 0 \\ 0 & -1 & 0 & 2 \\ 0 & 0 & 1 & 0 \\ 0 & 0 & 0 & 1\end{pmatrix}.$$
\end{example}

\begin{remark}
The map $t \mapsto (C_t, B_t)$ is called a (principal, tropical) \emph{$Y$-seed pattern} in \cite{SevenAMatrices}, 
as it is essentially equivalent to an assignment of tropical coefficient variables (usually denoted~$\bold y$) 
and $B$-matrix to each vertex $t\in \mathbb T_n$, cf.\ \cite[Proposition~3.6.5]{CAtextbook1-3}.  
\end{remark}

\begin{definition}
\label{def:redgreen}
A (mutable) vertex $v$ in a quiver $Q_t$ (or $\tilde Q_t$) is green (resp., red) if every entry of $\bold c_{v;t}$ is nonnegative (resp., nonpositive).
\end{definition}

\begin{example}
Every vertex in $Q_{t_0}$ is green.
Mutation at vertex $v$ toggles the color of $v$. 
(Other vertices may change color as well!)
\end{example}

\begin{theorem}[{\cite{QPs2signcoherence}}]
\label{thm:sign coherence}
Every vertex in $Q_t$ is either green or red (never both), regardless of the choice of initial quiver.
\end{theorem}

We next re-interpret Theorem~\ref{thm:sign coherence} using the following definition. 

\begin{definition}
\label{def:boxQ}
For  $t$ in $\mathbb T_n$, let $\boxQ_t$ denote the quiver obtained from $\tilde Q_t$ as follows: 
\begin{itemize}[leftmargin=.2in]
\item 
remove all frozen vertices and the arrows incident to them; 
\item
add a single new frozen vertex $\mathring v$; 
\item 
for each green mutable vertex~$v_j$, 
add $\sum_{i} c_{v_i' v_j; t}$ arrows  $\mathring{v} \rightarrow v_j$; 
\item
for each red mutable vertex~$v_j$, add $-\sum_{i} c_{v_i' v_j; t}$ arrows  $\mathring{v} \leftarrow v_j$.
\end{itemize}
\end{definition}

Note that, by Theorem~\ref{thm:sign coherence}, all arrows connecting frozen vertices to a given mutable vertex 
are oriented in the same direction. 
It also follows from Theorem~\ref{thm:sign coherence} that the ``bundling'' operation $\tilde Q_t \mapsto \boxQ_t$   
commutes with mutation: 

\begin{corollary}
\label{cor:Cmats and gluing}
If $t \stackrel{i}{\text{---}} t'$, or equivalently $\mu_{v_i}(\tilde Q_t)=\tilde Q_{t'}$, then 
 $\mu_{v_i}(\boxQ_t) = \boxQ_{t'}$.
\end{corollary}

\begin{remark} 
More generally, $C$-matrices can be used to compute the number of arrows between frozen and mutable vertices 
for  ``triangular extensions'' other than the principal framing,  
such as those that involve adding a frozen vertex $\mathring{v}$ and 
any number of arrows $\mathring{v} \rightarrow v$, cf.\ \cite[Theorem~3.2]{CaoLiCmats}.
\end{remark}


\begin{definition}
\label{def:underlying unoriented graph}
For a quiver $Q$ we define the underlying unoriented \emph{simple} graph~$K_Q$.
The graph $K_Q$ has the same vertex set as $Q$, and an edge $u - v$ whenever there are any arrows $u \rightarrow v$ or $v \rightarrow u$ in $Q$.
There are no parallel edges.

By a cycle in~$K_Q$, we mean a sequence of vertices in $K_Q$ where each consecutive pair is joined by an edge $\mathcal O = (w_0 - w_1 - \cdots - w_\ell=w_0)$, considered up to cyclic shifts.
In particular, cycles in~$K_Q$ have a direction of traversal. 
We will distinguish between two cycles with the same vertices traversed in the opposite order.
Cycles thus correspond to an element of the first homology group~$H_1(K_Q, \ZZ)$.
\end{definition}
\pagebreak[2]

\begin{definition}[{cf.\ \cite[Definition 2.10]{COQ}}]
\label{def:undirectedCycle}
A cycle $\mathcal O = (w_0 - w_1 - \cdots - w_\ell=w_0)$ in $K_Q$ is \emph{chordless} if there are no edges between $w_i$ and $w_{j}$ for $i \neq j \pm 1$.
If the arrows between $w_i$ and $w_{i+1}$ (in $Q$) are directed with the indexing of the cycle, $w_i \rightarrow w_{i+1}$, we say the cycle is \emph{forward-oriented}.
If instead the arrows are directed against the indexing, $w_i \leftarrow w_{i+1}$, then we say the cycle is \emph{backward-oriented}.
A cycle is \emph{oriented} if it is either forward-oriented or backward-oriented.

\hide{
While the cycle $\mathcal O$ lives in the unoriented simple graph $K_Q$, each edge in the cycle corresponds to one or more arrows in $Q$.
When the orientation of some (or all) of the arrows is known, we may replace the dashes ``$-$'' with ``$\rightarrow$'' or ``$\leftarrow$'' accordingly.
Thus if $\mathcal O$ is forward-oriented, then it could be denoted $(w_0 \rightarrow w_1 \rightarrow \cdots \rightarrow w_\ell = w_0)$.
}
\end{definition}

\begin{remark}
\label{rem:cycle terminology}
What we call a ``chordless cycle in $K_Q$'' is called a ``cycle'' in \cite{SevenAMatrices}. 
Thus, if a chordless cycle in $K_Q$ is oriented, then it is called an ``oriented cycle.'' 
Similarly, cycles which are not oriented are called "nonoriented." 

In \cite{SevenAMatrices}, cycles do not have an order of traversal, and so the forward/backward distinction does not arise.
The order of traversal of cycles plays an important role in~\cite{COQ} (cf.\ Definition~\ref{def:winding formula}).
\end{remark}

\begin{figure}[ht]
{
\newcommand\extht{1.5cm}
\newcommand\boxSize{2.5pt}
\includegraphics[alt={Two quivers on the vertices v1, v2, ..., v6. The quiver on the left is called Q, while the quiver on the right is called Q'. The quiver Q is an oriented cycles, with 3 arrows from v1 to v2 and one arrow from v2 to v3, v3 to v4, v4 to v5, v5 to v6, and from v6 to v1. There are no other arrows in Q. The quiver Q' instead has 3 arrows from v1 to v2, one arrow from v3 to v2, v3 to v6, v3 to v4, v4 to v5, v5 to v6, and v6 to v1. Thus Q' has an undirected cycle v1, v2, v3, v4, v5, v6, v1, but it has a chord from v3 to v6.}]{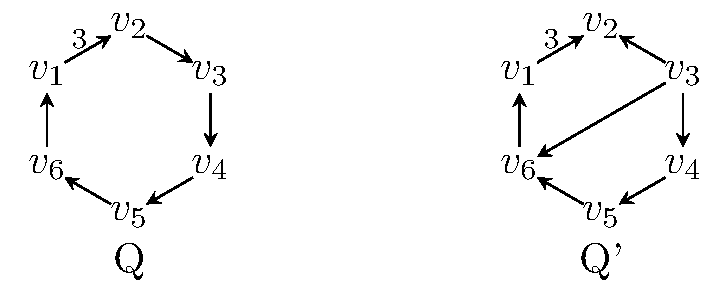}
} 

\caption{Two quivers $Q$ and $Q'$. The cycle $(v_1 - v_2 - v_3 - v_4 - v_5 - v_6 - v_1)$ in $K_Q$ is chordless and forward-oriented, while in $K_{Q'}$ it is not chordless.
The chordless cycles $(v_1 - v_2 - v_3 - v_6 - v_1)$ and $(v_3 - v_4 - v_5 - v_6 - v_3)$ in $K_{Q'}$ are not oriented.}
\label{fig:chordless cycles}
\end{figure}

\begin{definition}[{\cite{BGZsymmetrizable}}]
\label{def:quasi cartan}
Fix a quiver $Q$ with $B$-matrix $B_Q = (b_{vw})$. 
A \emph{quasi-Cartan companion} of~$Q$ is a symmetric matrix $A = (a_{vw})$ such that $a_{vv}=2$ and $a_{vw} = \pm b_{vw}$ otherwise. 
\end{definition}

A quasi-Cartan companion of a given matrix~$B_Q$ is determined by a choice of $\binom n 2$ signs.
The way these signs are chosen along chordless cycles will play a special role: 

\begin{definition}[{\cite{SevenAMatrices}}]
\label{def:admissible}
A quasi-Cartan companion $A=(a_{vw})$ of a quiver $Q$ is \emph{admissible} if for every chordless cycle $\mathcal O = (w_0 - \cdots - w_\ell=w_0)$ in~$K_Q$ the number $\# \{ i | a_{w_i w_{i+1}} > 0 \text{ and } i < \ell \}$ is odd when $\mathcal O$ is oriented, and is even otherwise. 
\end{definition}


\begin{definition}[{Cf.\ Definition~\ref{def:indexing}}]
\label{def:A t}
Fix an acyclic initial quiver $Q_0$, with exchange matrix $B_{Q_0} = (b_{vw})$.
We define the \emph{initial quasi-Cartan companion} $A_{t_0} = (a_{vw})$ by $a_{vv}=2$ and $a_{vw} = - |b_{vw}|$ for $v\neq w$.
For each $t \in \mathbb T_n$ we define the matrix $A_t = C_t^T A_{t_0} C_t$.
\end{definition}

The following theorems of A.~Seven will play an important role in our proof of Theorem~\ref{thm:acyclic totally proper}.

\begin{theorem}[{\cite[Theorem 1.3]{SevenAMatrices}}]
\label{thm:SevenAdmissibleCompanions}
Suppose the initial quiver $Q_0$ is acyclic. 
Then for every $t \in \mathbb T_n$, the matrix $A_t = C_t^T A_{t_0} C_t$ (see Definition~\ref{def:A t}) is an admissible quasi-Cartan companion of $Q_t$. 
\end{theorem}

We will also need the following result. 
\pagebreak[3]

\begin{theorem}[{\cite[Theorems 1.3, 1.4]{SevenAMatrices}}]
\label{thm:seven A properties}
Let $Q_0$ be acyclic, and fix a vertex $t \in \mathbb T_n$.
Let $A_t= (a_{uv;t})$. 
\begin{itemize}
\item For an arrow $u \rightarrow v$ in $Q_t$, we have $a_{uv;t} > 0$ if and only if $u$ is red and $v$ is green in $Q_t$.
\item Every oriented path of mutable vertices $w_1 \rightarrow \cdots \rightarrow w_m$ in $Q_t$ has at most one positive entry in $A_t$ (thus $a_{w_i w_{i+1};t} > 0$ for at most one $i < m$).
\end{itemize}
\end{theorem}

An illustration of Theorem~\ref{thm:seven A properties} can be found in Example~\ref{eg:Check signs of A t}. 

\begin{remark}
To see how \cite[Theorem 1.3]{SevenAMatrices} implies the first 
statement of Theorem~\ref{thm:seven A properties}, assume that $b_{ji} < 0$ 
and check the four combinations of $\sgn(\bold c_i)$ and~$\sgn(\bold c_j)$. 
\end{remark}

\begin{remark}
For quivers mutation equivalent to acyclic quivers with sufficiently many arrows, 
a number of properties of the matrices $B_t$ and $C_t$ matrices  
have been recently established by T.~Ervin, see \cite[Proposition~5.2, Lemma~5.4]{ErvinRedSize}. 
These results may potentially be used to obtain enhancements of Theorem~\ref{thm:seven A properties}.
\end{remark}

\begin{definition}[{\cite[Theorem 2.8.3]{CAtextbook1-3}}]
\label{def:mutation mat}
We define a pair of square matrices $M_Q(v, \pm 1)$, associated to mutating a quiver $Q$ at a given vertex $v$.
Let ${B_Q=(b_{pq})}$ be the exchange matrix of $Q$ and choose $\varepsilon = \pm 1$. 
Let $J$ be the diagonal matrix with $J_{vv} = -1$ and $J_{qq}=1$ for $w \neq v$.
Let $E = (e_{pq})$ be the matrix with $v$th column $e_{qv} = \max(0,-\varepsilon b_{qv})$ and $0$ otherwise.
Then $M_Q(v, \varepsilon) = J + E$. 
\end{definition}

The matrix $M_Q$ can be used to mutate the $B$ and $C$-matrices (and sometimes the quasi-Cartan companion) associated with $Q$.
Recall that red and green vertices are determined by the signs of the $C$-matrix (Definition~\ref{def:redgreen}).

\begin{proposition}[{ {\cite[Theorem 2.8.3]{CAtextbook1-3}}, cf.\ {\cite[Lemma 3.1]{SevenAMatrices}}}]
\label{prop:A matrix mutation acyclic}
Fix an initial quiver~$Q_0$.
Choose a quiver $\tilde Q_t \in [\widehat Q_0]$, and let $Q = Q_t$ (the mutable part of $\tilde Q_t$).
Select a vertex~$v_i$ of $Q$, and let~$t \stackrel{i}{\text{---}} t'$.
Then:
\begin{itemize}
\item for either value of $\varepsilon=\pm1$, 
\begin{equation}
\label{eq:Bcong}
B_{t'} = M_Q(v_i, \varepsilon) B_t M_Q(v_i, \varepsilon)^T;
\end{equation}
\item the $C$-matrices satisfy $$ C_{t'} = \begin{cases} C_t M_Q(v_i, 1)^T & \text{if $v_i$ is green,} \\ C_t M_Q(v_i, -1)^T & \text{if $v_i$ is red;} \end{cases}$$ 
\item if the initial quiver $Q_0$ is acyclic then 
$$A_{t'} = \begin{cases} M_Q(v_i, 1) A_t M_Q(v_i, 1)^T & \text{if $v_i$ is green,} \\ M_Q(v_i, -1) A_t M_Q(v_i, -1)^T  & \text{if $v_i$ is red.}\end{cases}$$ 
\end{itemize}
\end{proposition}

The latter two formulas, involving $C_t$ and $A_t$, follow from their definitions and~\eqref{eq:Bcong}.

\pagebreak[2]

\begin{example}
\label{eg:congruenceAB}
We demonstrate Proposition~\ref{prop:A matrix mutation acyclic}. 
Take our initial quiver~$Q_0$ to be the acyclic quiver with vertices and arrows $ v_1 \stackrel{2}{\rightarrow} v_3 \stackrel{}{\rightarrow} v_2$ and $v_1\stackrel{}{\rightarrow} v_2$. Let~$t \in \mathbb T_n$ satisfy 
$t_0 \stackrel{3}{\text{---}} t_1 \stackrel{1}{\text{---}} t_2 \stackrel{2}{\text{---}} t.$

Consider the matrices $B_t,$ and $C_t$, and their associated $A_t$ shown below.

$$B_t = \begin{pmatrix} 0 & 3 & -13 \\ -3 & 0 & 5 \\ 13 & -5 & 0 \end{pmatrix}, C_t = \begin{pmatrix} 8 & -3 & 0 \\ 3 & -1 & 0 \\ 3 & -1 & -1 \end{pmatrix},$$
$$A_t = \begin{pmatrix} 2 & -3 & 13 \\ -3 & 2 & -5 \\ 13 & -5 & 2 \end{pmatrix} = C_t^T \begin{pmatrix} 2 & -1 & -2 \\ -1 & 2 & -1 \\ -2 & -1 & 2 \end{pmatrix} C_t.$$
By inspecting the signs of (the columns of) $C_t$, we see that $v_1$ is green while~$v_2$ and~$v_3$ are red in~$Q_t$.

Let $t' \stackrel{1}{\text{---}}t$ (so $Q_{t'} = \mu_{v_1}(Q_t)$). 
We compute the mutation at $v_1$ and find
$$B_{t'} = 
\begin{pmatrix}
0 & -3 & 13 \\
3 & 0 & -34 \\
-13 & 34 & 0 
\end{pmatrix}.$$
The B-matrix $B_{t'}$ for $\mu_{v_1}(Q_t)$ can be computed from $B_t$ by performing either of two different congruences (corresponding to $M_{Q_t}(v_1, 1),$ and $M_{Q_t}(v_1, -1)$ respectively):
$$ 
\begin{pmatrix} 
-1 & 0 & 0 \\
3 & 1 & 0 \\
0 & 0 & 1
\end{pmatrix} B_t \begin{pmatrix}
-1 & 0 & 0 \\
3 & 1 & 0 \\
0 & 0 & 1
 \end{pmatrix}^T 
= \begin{pmatrix} 
-1 & 0 & 0 \\
0 & 1 & 0 \\
13 & 0 & 1
\end{pmatrix} B_t \begin{pmatrix} 
-1 & 0 & 0 \\
0 & 1 & 0 \\
13 & 0 & 1
\end{pmatrix}^T=B_{t'}.$$

However the same operations applied to $A_t$ do not both result in quasi-Cartan companions. 
In this case, because $v_1$ is green, congruence with $M_{Q_t}(v_1, 1)$ gives the quasi-Cartan companion $A_{t'}$:
$$ \begin{pmatrix} 
-1 & 0 & 0 \\
3 & 1 & 0 \\
0 & 0 & 1
 \end{pmatrix} \begin{pmatrix}  2 & -3 & 13 \\ -3 & 2 & -5 \\ 13 & -5 & 2 \end{pmatrix} 
 \begin{pmatrix} 
 -1 & 0 & 0 \\
3 & 1 & 0 \\
0 & 0 & 1 
\end{pmatrix}^T  =  
\begin{pmatrix} 
2 & -3 & -13 \\
-3 & 2 & 34 \\
-13 & 34 & 2 
\end{pmatrix}.$$ 
If instead we multiply with the matrix $M_{Q_t}(v_1, -1)$:
$$ \begin{pmatrix} 
-1 & 0 & 0 \\
0 & 1 & 0 \\
13 & 0 & 1
 \end{pmatrix} 
 \begin{pmatrix} 2 & -3 & 13 \\ -3 & 2 & -5 \\ 13 & -5 & 2
 \end{pmatrix} 
 \begin{pmatrix} 
 -1 & 0 & 0 \\
0 & 1 & 0 \\
13 & 0 & 1
 \end{pmatrix}^T= 
\begin{pmatrix} 
2 & 3 & -39 \\
3 & 2 & -44 \\
-39 & -44 & 678 
\end{pmatrix}.$$ 
Note in particular that one of the diagonal entries of the last matrix is not $2$.
\pagebreak[1]

If we instead consider $t'' \stackrel{3}{\text{---}}t$, 
then the associated $B$ and $C$ matrices are:
$$B_{t''} = 
\begin{pmatrix} 
0 & -62 & 13 \\
62 & 0 & -5 \\
-13 & 5 & 0
\end{pmatrix}, C_{t''} = 
\begin{pmatrix} 8 & -3 & 0 \\ 3 & -1 & 0 \\ 3 & -6 & 1 \end{pmatrix}.$$
In this case, only congruence with $M_{Q_t}(v_3, -1)$ gives a quasi-Cartan companion:
$$ \begin{pmatrix} 
1 & 0 & 0 \\
0 & 1 & 5 \\
0 & 0 & -1 
\end{pmatrix} A_t 
\begin{pmatrix} 
1 & 0 & 0 \\
0 & 1 & 5 \\
0 & 0 & -1 
\end{pmatrix}^T =  \begin{pmatrix} 
2 & 62 & -13 \\
62 & 2 & -5 \\
-13 & -5 & 2
\end{pmatrix}. $$
\end{example}

\begin{example}
\label{eg:Check signs of A t}
Taking $A_t, B_t$ and $Q_t$ as in Example~\ref{eg:congruenceAB}, we see the only positive off-diagonal entry of $A_t$ is $a_{v_1v_3;t} = 13$.
As predicted by Theorem~\ref{thm:seven A properties} we see $v_3 \rightarrow v_1$, $v_3$ is red, and $v_1$ is green in $Q_t$.
\end{example}

\newpage

\section{Cyclically ordered quivers}
\label{sec:coqs}

We recall many definitions, notations, and results from \cite{COQ}. 
We include some new examples, but do not repeat proofs.

\begin{definition}
\label{def:cyclic ordering}
A \emph{cyclic ordering} of a set $V$ is a linear ordering considered up to cyclic shifts---that is, taking the minimal element $v$ and forming a new linear ordering where $v > u$ for all $u \neq v$ (the order is otherwise unchanged). 
By repeated cyclic shifts, any element can be made maximal or minimal.
The cyclic ordering associated to a linear ordering $v_1 < \cdots < v_n$ will be denoted $(v_1, \ldots, v_n)$.
%
\end{definition}

\begin{definition}
\label{def:coq}
A \emph{cyclically ordered quiver} $(Q, \sigma)$ (abbreviated COQ) is a quiver~$Q$ equipped with a cyclic ordering of its (mutable) vertices $\sigma$. 
Likewise, a \emph{linearly ordered quiver} is a quiver equipped with a linear ordering of its (mutable) vertices.
We may omit the ordering when it is clear from context or not needed, and simply denote a COQ by $Q$.
There are $n$ linearly ordered quivers associated to each COQ, given by ``tearing'' the cyclic order between different consecutive vertices in the cyclic ordering.

We say an oriented path $u \rightarrow v \rightarrow w$ makes a \emph{right turn} at $v$ if $u < v < w$ in some linear order associated to $\sigma$.
\end{definition}

{ 
\newcommand\cyclicLabels[7]{%
	\filldraw[black] (#4,#5)++(135:0.8cm) circle (2pt) node[above left=-1pt] {$v_1$} coordinate (v_1);
	\filldraw[black] (#4,#5)++(45:0.8cm) circle (2pt) node[above right=-1pt] {$#1$} coordinate (#1);
	\filldraw[black] (#4,#5)++(-45:0.8cm) circle (2pt) node[below right=-1pt] {$#2$} coordinate (#2);
	\filldraw[black] (#4,#5)++(-135:0.8cm) circle (2pt) node[below left=-1pt] {$#3$} coordinate (#3);
	\draw[black, dashed, decoration={markings, mark=at position 0 with {\arrow{<}}}, postaction={decorate}] (#4,#5) circle (0.8cm);
	\draw[black, -{stealth}, shorten >=3pt, shorten <= 3pt] (v_1) -- (v_2);
	\draw[black, -{stealth}, shorten >=3pt, shorten <= 3pt] (v_2) -- (v_4) node [#6] {{\tiny $2$}};
	\draw[black, -{stealth}, shorten >=3pt, shorten <= 3pt] (v_3) -- (v_4);
	\draw[black, -{stealth}, shorten >=3pt, shorten <= 3pt] (v_1) -- (v_3);
	\draw[black, -{stealth}, shorten >=3pt, shorten <= 3pt] (v_1) -- (v_4) node [#7] {{\tiny $3$}};
	\draw (#4,#5-1) node[below] {{\small $(v_1, #1, #2, #3)$}};
}
\begin{figure}[ht]
\newcommand\hgap{2.3cm}
\includegraphics[alt={Six distinct cyclically ordered quivers with the same underlying quiver, which is described in the caption. Each is drawn with the vertices placed on a dashed circle, with that order written below the cyclically ordered quiver. From left to right the orders are (v1, v2, v3, v4), (v1, v3, v2, v4), (v1, v2, v4, v3), (v1, v3, v4, v2), (v1, v4, v3, v2), (v1, v4, v2, v3). }]{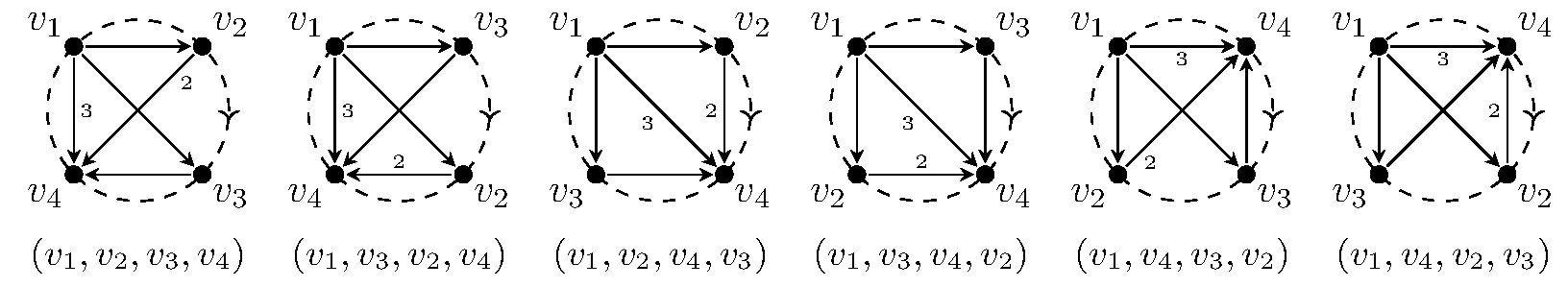}

\caption{The $4!/4 = 6$ distinct COQs on the $4$-vertex acyclic quiver with vertices and arrows $v_1 \stackrel{}{\rightarrow} v_2 \stackrel{2}{\rightarrow} v_4$, $v_1 \stackrel{}{\rightarrow} v_3 \stackrel {}{\rightarrow} v_4$, and $v_1 \stackrel 3 {\rightarrow} v_4$.}
\label{fig:cyclic orderings}
\end{figure}

\begin{definition}
\label{def:unipotent companion}
The \emph{unipotent companion} of a linearly ordered quiver $Q$ is the unipotent upper triangular matrix $U$ such that $ U^T - U = B_Q$, where the rows and columns of $B_Q$ are written in the linear ordering. 
The matrix $U$ can be constructed by negating $B_Q$, then setting the lower triangular part to $0$ and the diagonal entries to~$1$.
For a COQ~$Q$, the unipotent companion of any associated linearly ordered quiver is a unipotent companion of $Q$.
\end{definition}

\begin{example}
\label{eg:unipotents}
Consider the leftmost COQ $Q$ depicted in Figure~\ref{fig:cyclic orderings}.
This COQ has~$4$ associated linearly ordered quivers. We list their respective unipotent companions and linear orderings:
$$\begin{pmatrix} 1 & -1 &  -1 & -3 \\ 0 & 1 & 0 & -2 \\ 0 & 0 & 1 & -1\\ 0 & 0 & 0 & 1 \end{pmatrix} 
\quad \begin{pmatrix} 1 & 0 &  -2 & 1 \\ 0 & 1 & -1 & 1 \\ 0 & 0 & 1 & 3\\ 0 & 0 & 0 & 1 \end{pmatrix} 
\quad \begin{pmatrix} 1 & -1 &  1 & 0 \\ 0 & 1 & 3 & 2 \\ 0 & 0 & 1 & -1\\ 0 & 0 & 0 & 1 \end{pmatrix}
\quad \begin{pmatrix} 1 & 3 &  2 & 1 \\ 0 & 1 & -1 & -1 \\ 0 & 0 & 1 & 0\\ 0 & 0 & 0 & 1 \end{pmatrix} $$
$$ \hspace{0.15cm} v_1< v_2< v_3<v_4 \quad \quad v_2< v_3<v_4 < v_1 \hspace{0.6cm} v_3<v_4 < v_1< v_2 \quad  \quad v_4 < v_1< v_2< v_3$$
\hide{
\begin{tabular}{ c c c c }
$\begin{pmatrix} 1 & -1 &  -1 & -3 \\ 0 & 1 & 0 & -2 \\ 0 & 0 & 1 & -1\\ 0 & 0 & 0 & 1 \end{pmatrix}$ &
$\quad \begin{pmatrix} 1 & 0 &  -2 & 1 \\ 0 & 1 & -1 & 1 \\ 0 & 0 & 1 & 3\\ 0 & 0 & 0 & 1 \end{pmatrix}$ & 
$\quad \begin{pmatrix} 1 & -1 &  1 & 0 \\ 0 & 1 & 3 & 2 \\ 0 & 0 & 1 & -1\\ 0 & 0 & 0 & 1 \end{pmatrix}$ &
$\quad \begin{pmatrix} 1 & 3 &  2 & 1 \\ 0 & 1 & -1 & -1 \\ 0 & 0 & 1 & 0\\ 0 & 0 & 0 & 1 \end{pmatrix}$ \\
$a < b < c < d$ & $b < c < d < a$ & $c < d < a < b$  & $d < a < b < c$
\end{tabular}
}
\end{example}

\begin{proposition}[{\cite[Proposition 4.4]{COQ}}]
\label{prop:cyclic shift of U}
All unipotent companions of a COQ are integrally congruent.
\end{proposition}

Recall that the integral congruence class of a unipotent companion $U$ is the set of matrices $\{ G U G^T | G \in \GL_n(\ZZ)\}$.

\begin{definition} 
\label{def:wiggle}
A \emph{wiggle} is a transformation of a COQ which fixes the quiver $Q$, and transposes two adjacent vertices in the cyclic ordering that are not adjacent in $Q$.
We say that two COQs are wiggle equivalent if they have the same quiver, and their cyclic orderings are related by a sequence of wiggles. 
The wiggle equivalence class of a COQ is the set of all wiggle equivalent COQs, see Figure~\ref{fig:wiggle classes}.
If $U$ is a unipotent companion of some COQ in a wiggle equivalence class, we say that $U$ is a unipotent companion of the class.
\end{definition}

\begin{figure}[ht]
\newcommand\hgap{2cm}
\newcommand\h{1.7cm}
\includegraphics[alt={The COQs from Figure~\ref{fig:cyclic orderings} arranged according to their wiggle equivalence classes. In the top left, we see the COQs with orderings (v1, v2, v3, v4) and (v1,v3,v2,v4), related by a wiggle at (v2 v3). Note that v2 and v3 have no arrows between them, but are adjacent in these cyclic orderings. In the top right, the COQ with cyclic ordering (v1, v2, v4, v3). Note that each adjacent pair of vertices on the cycle are also adjacent in the quiver, so there are no wiggles possible. In the bottom left we have the COQ with ordering (v1, v3, v4, v2). Again no wiggles are possible. Finally in the bottom right we have the COQs with cyclic orderings (v1, v4, v3, v2) and (v1, v4, v2, v3). These COQs are related by a wiggle at (v2 v3).}]{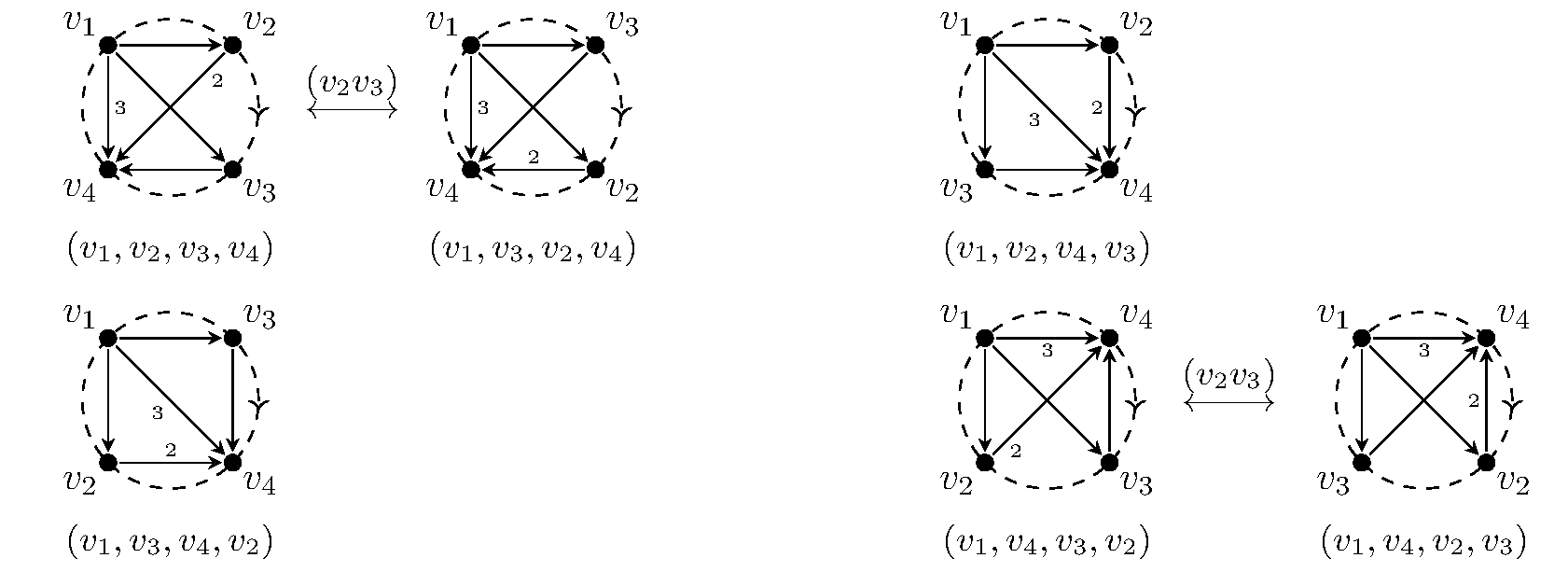}

\caption{The $4$ wiggle equivalence classes of the COQs in Figure~\ref{fig:cyclic orderings}.
}
\label{fig:wiggle classes}
\end{figure}
} 

\begin{proposition}[{\cite[Proposition 4.5]{COQ}}]
\label{prop:wiggle preserves U}
All unipotent companions of wiggle equivalent COQs are integrally congruent.
\end{proposition}

\begin{definition} 
\label{def:winding formula}
Fix a COQ $(Q, \sigma)$.
Let $\mathcal O = (w_0 - w_1 - \cdots - w_k=w_0)$ be a cycle in $K_Q$ (Definition~\ref{def:undirectedCycle}).
Fix a linear order $<$ associated to~$\sigma$.
The \emph{winding number} $\wind_\sigma(\mathcal O)$ is the (signed) number of times we ``wrap around'' the linear order: 
\begin{equation}
\label{eq:wind formula}
\wind_\sigma(\mathcal O) = \# \{ w_i | w_i \rightarrow w_{i+1}, w_i >  w_{i+1} \} -\#\{ w_i | w_i \leftarrow w_{i+1}, w_i <  w_{i+1}  \}.
\end{equation}
\end{definition}

We omit the (straightforward) proofs that $\wind_\sigma(\mathcal O)$ does not depend on the choice of linear order associated to $\sigma$, and that this definition agrees with \cite[Definition~2.10]{COQ}.

\begin{example}
Let $Q$ be the underlying quiver of the COQs shown in Figure~\ref{fig:wiggle classes}, and
consider the (chordless) cycle $\mathcal O = (v_1 - v_2 - v_4 - v_1)$ in $K_Q$.
The winding number of this undirected cycle is $0$ in the three COQs in the top row and $1$ in the three COQs in the bottom row. 
\end{example}

\begin{remark}
\label{rem:winding number topology}
The winding number corresponds to the winding number of certain maps from the cycle (as a cell complex) to the circle.
Specifically, map the vertices onto the circle according to the cyclic ordering (\emph{not} the order of traversal), and then map each edge to an arc between the endpoints, starting from the tail of an associated arrow and traveling clockwise to the head. See Figure~\ref{fig:map to S1}.
\end{remark}

\begin{figure}[ht]
\includegraphics[alt={On the left, a COQ with quiver whose arrows are v1 to v2, v2 to v4, v1 to v3, v3 to v4, and whose cyclic ordering is (v1, v2, v3, v4).}]{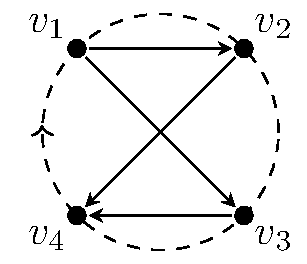}
\quad \quad \quad
\includegraphics[alt={On the right, the same vertices drawn on a dashed circle in the same positions as the COQ, but the arrows have been drawn along the dashed circle. The arrow from v1 to v2 now lies on the circular arc clockwise from the location of v1 (135 degrees on the circle) to the location of v2 (45 degrees). The arrow from v2 to v3 lies on the circular arc clockwise from the location of v2 to the location of v4 (-135 degrees). The arrow from v3 to v4 lies on the circular arc clockwise from the location of v3 (-45 degrees) to the location of v4. The arrow from v1 to v3 lies on the circular arc clockwise from the location of v1 to the location of v3. In particular, none of the arrows go through the point at 180 degrees, while they cover each point from 135 degrees to -135 degrees (traveling clockwise) twice.}]{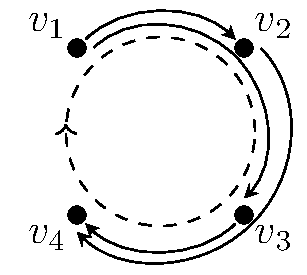}

\caption{
A COQ $Q$ and the image of the cycle $\mathcal O = (v_1 - v_2 - v_4 - v_3 - v_1)$ in $K_Q$, with the arrows drawn on the circle. 
In this case the image is contractible, and so the winding number is $0$. 
}
\label{fig:map to S1}
\end{figure}

\begin{theorem}[{\cite[Theorem 2.14]{COQ}}]
\label{th:wiggle/winding}
Two COQs $(Q, \sigma)$ and $(Q, \sigma')$ are wiggle equivalent if and only if $\wind_{\sigma}(\mathcal O) = \wind_{\sigma'}(\mathcal O)$ for every undirected cycle $\mathcal O$ in $K_Q$.
\end{theorem}
\pagebreak[2]

We now turn to mutations of COQs.

\begin{definition} 
\label{def:proper}
A vertex $v$ in a COQ $(Q, \sigma)$ is \emph{proper} if every oriented path ${u \rightarrow v \rightarrow w}$ makes a right turn at $v$ (Definition~\ref{def:coq}).
The wiggle equivalence class is \emph{proper} if for every vertex $v$ there is a COQ in the wiggle equivalence class in which $v$ is proper.
\end{definition}


\begin{example}
In the wiggle equivalence classes shown in Figure~\ref{fig:wiggle classes}, vertices $v_1, v_4$ are proper vertices in each class, as they have no oriented paths through them. 
Vertex~$v_2$ is proper only in the three COQs in the top row, while vertex $v_3$ is proper only in the three COQs in the leftmost column(s).
Thus the wiggle equivalence class in the top left, using the cyclic orderings $(v_1, v_2, v_3, v_4)$ and $(v_1, v_3, v_2, v_4)$, is the only class which is proper.
\end{example}

\begin{definition}  
To mutate a COQ $Q$ at a proper vertex $v$, mutate the quiver as usual and reposition $v$ in the cyclic ordering so that $v$ is still proper (that is, move $v$ clockwise in the cyclic ordering past all the elements of $\In(v)$ in the mutated quiver, without passing any vertices in $\Out(v)$; 
all choices give wiggle equivalent COQs).
We denote the resulting wiggle equivalence class by $\mu_v(Q)$.
The \emph{proper mutation class} of a COQ is the set of all COQs which can be obtained from the original by a sequence of proper mutations and wiggles.
\end{definition}

{ 
\newcommand\cyclicVerts[5]{%
	\filldraw[black] (#4,#5)++(135:0.8cm) circle (2pt) node[above left=-1pt] {$v_1$} coordinate (v_1);
	\filldraw[black] (#4,#5)++(45:0.8cm) circle (2pt) node[above right=-1pt] {$#1$} coordinate (#1);
	\filldraw[black] (#4,#5)++(-45:0.8cm) circle (2pt) node[below right=-1pt] {$#2$} coordinate (#2);
	\filldraw[black] (#4,#5)++(-135:0.8cm) circle (2pt) node[below left=-1pt] {$#3$} coordinate (#3);
	\draw[black, dashed, decoration={markings, mark=at position 0 with {\arrow{<}}}, postaction={decorate}] (#4,#5) circle (0.8cm);
}
\begin{figure}[ht]
\includegraphics[alt={A sequence of three proper mutations between four COQs. Starting on the left, the quiver has vertices v1, v2, v3, v4 (with that cyclic ordering), and an arrow from v2 to v1, v3 to v1, v3 to v4, as well as two arrows from v2 to v4 and 3 arrows from v4 to v1. In particular, v1 is a sink vertex. Thus there are no oriented paths through v1, and so v1 is proper.}]{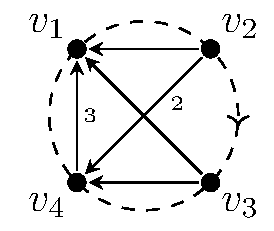}
\raisebox{1cm}{$\stackrel{\mu_{v_1}}{\longleftrightarrow}$}
\includegraphics[alt={The first mutation is at v1, which only reverses the arrows touching v1, and leaves the cyclic ordering unchanged. Notice that the only oriented path of length two through v2 starts at v1 and ends at v4. This is a right turn, so v2 is proper. }]{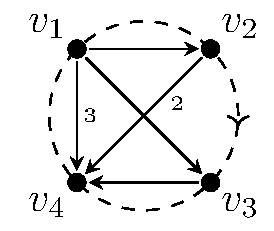}
\raisebox{1cm}{$\stackrel{\mu_{v_2}}{\longleftrightarrow}$}
\includegraphics[alt={We mutate at v2 to arrive at a new COQ with cyclic ordering (v1, v3, v4, v2) and an arrow from v1 to v3, v3 to v4, v2 to v1, as well as five arrows from v1 to v4 and two arrows from v4 to v2. As v2 had an oriented path through it, we have changed its position in the cyclic ordering. }]{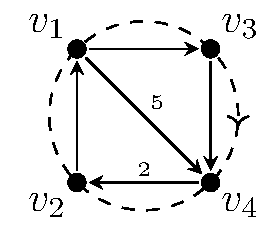}
\raisebox{1cm}{$\stackrel{\mu_{v_3}}{\longleftrightarrow}$}
\includegraphics[alt={We next wish to mutate at v3. It is proper as the only oriented 2-path goes from v1 to v4, and v1, v3, v4 is a right turn. Proper mutation at v3 results in the COQ with cyclic ordering (v1, v4, v3, v2) and a single arrow from v2 to v1, v3 to v1, v4 to v3, as well as two arrows from v4 to v2 and six arrows from v1 to v4.}]{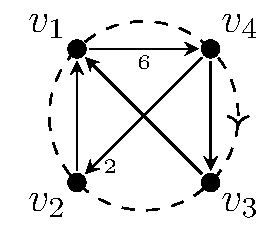}

\caption{A sequence of proper mutations.}
\label{fig:proper mutations}
\end{figure}
}

\begin{proposition}[{\cite[Proposition 6.4]{COQ}}]
\label{prop:proper well def}
Proper mutation is well defined up to wiggle equivalence---that is, if two COQs are wiggle equivalent and a vertex $v$ is proper in each, then performing the proper mutation $\mu_v$ to each results in the same wiggle equivalence class of COQs.
\end{proposition}

\begin{theorem}[{\cite[Theorem 7.1]{COQ}}]
\label{thm:payoff}
The unipotent companions of COQs which are related by a proper mutation are integrally congruent.
\end{theorem}

\begin{definition}[{\cite[Definition 14.1]{COQ}}]
A (wiggle equivalence class of) COQ(s) is \emph{totally proper} if every COQ in its proper mutation class is proper.
We say a cyclic ordering $\sigma$ is \emph{totally proper} for a quiver $Q$ if $(Q,\sigma)$ is totally proper.
If $Q$ has a totally proper cyclic ordering, we say $Q$ is \emph{totally proper}.
\end{definition}



Certain unipotent companions are transformed by matrix conjugation with the \emph{same} matrices (see Definition~\ref{def:mutation mat}) that transform the $B$-matrix.
One case was established in \cite{COQ}, and follows.
We generalize this result in the next section.

\begin{lemma}[{\cite[Lemma 7.2 and Remark 7.3]{COQ}}]
\label{lem:mutates like B v1}
Suppose that $v_j$ is a proper vertex in the COQ $Q$.
Let $U$ be a unipotent companion of $Q$ with linear order $<$ such that $v_i < v_j$ for $v_i \in \In(v_j)$ and $v_i < v_k$ for $v_k \in \Out(v_j)$.
Then the matrix $M_Q(v_j, -1) U M_Q(v_j, -1)^T$ is a unipotent companion of $\mu_{v_j}(Q)$.
\end{lemma}

We conclude this section by recalling two easily expressed invariants of COQs, along with some examples.
Definition~\ref{def:alexander and markov} and Proposition~\ref{prop:markov positive} do not play any role in the proof of Theorem~\ref{thm:acyclic totally proper}, but do appear prominently in applications and corollaries. 

\begin{definition}[{\cite[Definition 8.1]{COQ}}]
\label{def:alexander and markov}
Given a COQ $Q$ with unipotent companion $U$, the \emph{Alexander polynomial} of $Q$ is given by~$\Delta_Q(t) = \det(tU - U^T)$.
The \emph{Markov invariant} $M_Q$ is defined by~$M_Q = n + (\text{coefficient of } t^{n-1} \text{ in } \Delta_Q(t))$.
\end{definition}

\begin{remark}
As the exchange matrix $B$ and unipotent companion $U$ of a linearly ordered quiver $Q$ are related by $B = U^T - U$, it follows that $$\det(B) = \det(U^T - U) = (-1)^n \Delta_Q(1).$$
\end{remark}

\begin{proposition}
\label{prop:markov positive}
Let $Q$ be an $n$-vertex acyclic quiver.
Fix a linear order $<$ of the vertices of $Q$ so that $ v_i < v_j$ whenever $v_i \rightarrow v_j$.
Let $U$ be the unipotent companion of the linearly ordered quiver $(Q, <)$.
For a cycle $\mathcal O = (w_0 - \cdots - w_\ell = w_0)$ in $K_Q$, let $wt(\mathcal O) = \prod_i |b_{w_i w_{i+1}}|$ if there is exactly one location $i$ with $w_{i} \leftarrow w_{i+1}$ (thus $w_j \rightarrow w_{j+1}$ for all $j\neq i$), and $wt(\mathcal O) = 0$ otherwise.
Then the Markov invariant 
$$M_Q = \sum_{v_i < v_j} b_{v_i v_j}^2 + \sum_{\mathcal O} wt(\mathcal O).$$
\end{proposition}

\begin{proof}
The $t^{n-1}$ coefficient of $\det(tU - U^T)$ is negative of the trace of $U^{T} U^{-1}$. 
Let $N = I-U$, note $N$ is positive and strictly upper triangular. 
Thus $$U^{-1} = I + N + N^2 + \cdots + N^{n-1},$$ and so $\tr(U^{T}\! U^{-1}) = \sum_{k=0}^n \tr(U^T\! N^k)$.
We compute $\tr(U^T\! I)\! = n$, $\tr(U^T\! N)\! =-\sum_{v_i < v_j}\! b_{v_i v_j}^2$ and $\tr(U^T\! N^k) = -\sum_{\mathcal O} wt(\mathcal O),$ where the sum is over the cycles $\mathcal O$ of length $k$. 
(Recall that cycles are considered up to cyclic shifts.)
\end{proof}

\begin{example}
\label{eg:Markov}
Let $Q$ be the leftmost COQ shown in Figure~\ref{fig:cyclic orderings}.
Then $$\Delta_Q(t)= t^{4} + 21 t^{3} - 43 t^{2} + 21 t + 1,$$ and $M_Q = 1^2 + 1^2 + 1^2 + 2^2 + 3^2 + 1 \cdot 2 \cdot 3 + 1 \cdot 1 \cdot 3 = 4 + 21 = 25.$
\end{example}

\hide{\begin{remark}
\label{rem:Schwartz}
Amanda Schwartz has recently developed a combinatorial description for \emph{all} of the coefficients of the Alexander polynomial $\Delta_Q(t)$ in the case that $Q$ is a tree quiver (that is, $K_Q$ is a tree) \cite[Theorem 4.2]{SchwartzTree}.
\end{remark}
}
Further examples appear in Section~\ref{sec:cor}.

\newpage

\section{Technical Lemmas}
\label{sec:tech}

In this section we establish lemmas and propositions which we will need in the proof of Theorem~\ref{thm:acyclic totally proper}.
We begin with generalizations of Lemma~\ref{lem:mutates like B v1}.

\begin{lemma}
\label{lem:mutates like B v1gen}
Suppose that $v_j$ is a proper vertex in the COQ $Q$.
Let $U=(u_{v_i v_k})$ be a unipotent companion of $Q$ with linear order $<$ such that $v_i < v_j$ for $v_i \in \In(v_j)$.
Then the matrix $M_Q(v_j, -1) U M_Q(v_j, -1)^T$ is a unipotent companion of $\mu_{v_j}(Q)$.
\end{lemma}

\begin{proof}
We follow the proof of \cite[Lemma 7.2]{COQ} closely. 
Throughout $\In(v_j)$ and~$\Out(v_j)$ respectively refer to the inset and outset of $v_j$ in $Q$.

For the mutated quiver $Q' = \mu_{v_j}(Q)$ let $<'$ be a linear order such that 
\begin{itemize}
\item $v_j <' v_i$ for all $v_i \in \In(v_j)$, 
\item if both $v_k \in \Out(v_j)$ and $v_k < v_j$ then $v_k <' v_j$, and 
\item otherwise $<'$ agrees with~$<$ (so we have only changed the relative position of $v_j$).
\end{itemize}
Such an order exists because $v_j$ is proper in $Q$.
Let $U'=(u'_{v_i v_k})$ be the unipotent companion of $Q'$ with linear ordering $<'$.

Let $B_Q = (b_{v_iv_k})$, set $n_{v_iv_j} = -\max(0, b_{v_iv_j})$ and also 
$$\delta_{v_i} = \begin{cases} -1 & \text{if } v_i=v_j; \\ 1 & \text{else.} \end{cases}$$

We wish to show that:
\begin{equation} 
\label{eq:mutate M -1}
u'_{v_iv_k} = \delta_{v_i} \delta_{v_k} u_{v_i v_k} - n_{v_i v_j} u_{v_j v_k} \delta_{v_k} - \delta_{v_i} u_{v_i v_j} n_{v_k v_j} + n_{v_i v_j} n_{v_k v_j}.
\end{equation}
Let $u''_{v_i v_k}$ denote the right hand side of \eqref{eq:mutate M -1}.

By construction, we have 
\begin{equation*}
u'_{v_i v_k} = \begin{cases} 
1 & \text{if } v_i=v_k; \\
 b_{v_j v_k} & \text{if } v_i=v_j <' v_k; \\ 
 b_{v_i v_j} & \text{if } v_i <' v_j=v_k; \\
 -b_{v_i v_k} - n_{v_i v_j} n_{v_k v_j} + \max(0, -b_{v_i v_j}) \max(0, -b_{v_j v_k}) & 
 \text{if }  v_j \neq v_i <' v_k \neq v_j; \\ 
 0 & \text{if } v_i >' v_k.
\end{cases}
\end{equation*}

Every case where both $v_j \leq' v_i,$ and $v_j \leq' v_k$ appears in \cite[Lemma 7.2]{COQ}, and the arguments therein apply.
We consider the cases where $v_j$ appears after $v_i$ or $v_k$ in~$<'$.
Note by construction of $<'$ that if $w <' v_j$ then $u_{v_jw} = 0 = n_{wv_j}$ and~$b_{wv_j} = -u_{wv_j}.$

\begin{itemize}
\item If $v_i = v_k <' v_j$ then $u'_{v_i v_i} = 1 - 0 - 0 + 0 = u''_{v_i v_i}.$
\item If $v_i <' v_j=v_k$ then, 
 $u'_{v_i v_j} = b_{v_i v_j} = -u_{v_i v_j} - 0 - 0 + 0 = u''_{v_i v_j}$.

\item We split $v_j \neq v_i <' v_k \neq v_j$ into $2$ cases, depending on the relative position of $v_j$.
\begin{itemize}
\item If $v_i <' v_j <' v_k$ then 
 $n_{v_k v_j} = -\max(0, b_{v_k v_j})$ and $u_{v_i v_j}= \max(0, -b_{v_i v_j})$.
So
$$u'_{v_i v_k} =  -b_{v_i v_k} + \max(0, -b_{v_i v_j}) \max(0, -b_{v_j v_k}) = u_{v_i v_k} - 0 - u_{v_i v_j} n_{v_k v_j} + 0 = u''_{v_i v_k}.$$
\item
If $ v_i <' v_k <' v_j$,
then
$u'_{v_i v_k} = -b_{v_i v_k} = u_{v_i v_k} - 0 - 0 + 0 = u''_{v_i v_k}.$
\end{itemize}
\item If $v_k <' v_i$ and $v_k <' v_j$ then $u'_{v_i v_k} = 0 - 0 - 0 + 0 = u''_{v_i v_k}.$ 
\end{itemize}

Let $\pi$ be the permutation matrix which transforms $<$ into $<'$. 
Equation~\eqref{eq:mutate M -1} can be rewritten as $U' = \pi^T M_Q(v_j, -1) U M_Q(v_j, -1)^T \pi$ (recall the matrices $E$ and $J$ defined in Definition~\ref{def:mutation mat}).
The claim follows.
\end{proof}


We establish a corresponding version of Lemma~\ref{lem:mutates like B v1gen} for $\varepsilon = 1$.

\begin{lemma}
\label{lem:mutates like B v2gen}
Suppose that $v_j$ is a proper vertex in the COQ $Q$.
Let $U$ be a unipotent companion of $Q$ with linear order $<$ such that $v_j < v_k$ for $v_k \in \Out(v_j)$.
Then the matrix $M_Q(v_j, 1) U M_Q(v_j, 1)^T$ is a unipotent companion of $\mu_{v_j}(Q)$.
\end{lemma}

\begin{proof}
We reduce to Lemma~\ref{lem:mutates like B v1gen} by introducing the \emph{opposite} linearly ordered quiver~$R^{op}$ of another linearly ordered~$R$.
The quiver of $R^{op}$ has the same vertices as $R$, but each arrow is reversed ($v \rightarrow w$ becomes $v \leftarrow w$) and the linear ordering is reversed ($v < w$ becomes $v > w$).
It is easy to see that if $U_R$ is a unipotent companion of $R$, then $U_R^T$ is a unipotent companion of $R^{op}$ (note that $U_R^T$ is upper triangular in the reversed linear order).
We also have $M_R(v_j, 1) = M_{R^{{op}}}(v_j, -1)$. 

Now we compute:
\begin{align*}
M_Q(v_j, 1) U M_Q(v_j, 1)^T &= ( M_Q(v_j, 1) U^T\! M_Q(v_j, 1)^T )^T\! \\
&= ( M_{Q^{op}}(v_j, -1) U^T\! M_{Q^{op}}(v_j, -1)^T )^T.
\end{align*} 

Lemma~\ref{lem:mutates like B v1gen} implies that $M_{Q^{op}}(v_j, -1) U^T M_{Q^{op}}(v_j, -1)^T$ is a unipotent companion of $Q^{op}$. 
Thus its transpose is a unipotent companion of $Q$, as desired.
\hide{
Note that, up to re-indexing, the opposite quiver switches the sign of $\varepsilon$.
Let $U$ be the unipotent companion of $Q$. 
Let $U'$ be the unipotent companion of $\mu_k(Q)$ \cS{with what linear order?}.
Thus by Lemma~\ref{lem:mutates like B v1} and Proposition~\ref{prop:op companion} (twice):
$U' = w_0 (G' w_0 U^T w_0 G'^T)^T w_0 = (w_0 G' w_0) U (w_0 G' w_0)^T$
which we recognize as the other mutating matrix.
\cS{there are some easy observations to spell this out. Do it all in this proof, no need for these one-use defs.}
}
\end{proof}

\begin{example}
Let $Q$ be the leftmost COQ depicted in Figure~\ref{fig:cyclic orderings}.
The unipotent companions of $Q$ appear in Example~\ref{eg:unipotents}, let $U$ be the rightmost unipotent companion corresponding to the linear order $v_4 < v_1 < v_2 < v_3$.
Vertex $v_2$ is proper (and both $Q$ and $\mu_{v_2}(Q)$ appear in the center of Figure~\ref{fig:proper mutations}).
We compute the unipotent companion $M_Q(v_2, -1) U M_Q(v_2, -1)^T$ of $\mu_{v_2}(Q)$:
\begin{align*}  
\begin{pmatrix} 
1 & 0 & 0 & 0 \\
0 & 1 & 1 & 0 \\
0 & 0 & -1 & 0 \\
0 & 0 & 0 & 1
\end{pmatrix} 
\begin{pmatrix} 1 & 3 &  2 & 1 \\ 0 & 1 & -1 & -1 \\ 0 & 0 & 1 & 0\\ 0 & 0 & 0 & 1 \end{pmatrix}
\begin{pmatrix}
1 & 0 & 0 & 0 \\
0 & 1 & 1 & 0 \\
0 & 0 & -1 & 0 \\
0 & 0 & 0 & 1
\end{pmatrix}^T
&= \begin{pmatrix}
1 & 5 & -2 & 1 \\
0 & 1 & 0 & -1 \\
0 & -1 & 1 & 0 \\
0 & 0 & 0 & 1
\end{pmatrix}.
\end{align*}
This matrix is upper triangular with the linear order $v_4 < v_2 < v_1  < v_3$ (thus we have swapped the $2$nd and $3$rd rows and columns), in agreement with Lemma~\ref{lem:mutates like B v1gen}. 

If we instead take $U'$ to be the leftmost unipotent companion depicted in Example~\ref{eg:unipotents} (linear order $v_1 < v_2 < v_3 < v_4$), then we may apply Lemma~\ref{lem:mutates like B v2gen}.
We find another unipotent companion $M_Q(v, 1) U' M_Q(v, 1)^T$:
\begin{align*}
\begin{pmatrix}
1 & 0 & 0 & 0 \\
0 & -1 & 0 & 0 \\
0 & 0 & 1 & 0 \\
0 & 2 & 0 & 1
\end{pmatrix}
\begin{pmatrix} 1 & -1 &  -1 & -3 \\ 0 & 1 & 0 & -2 \\ 0 & 0 & 1 & -1\\ 0 & 0 & 0 & 1 \end{pmatrix} 
\begin{pmatrix}
1 & 0 & 0 & 0 \\
0 & -1 & 0 & 0 \\
0 & 0 & 1 & 0 \\
0 & 2 & 0 & 1
\end{pmatrix}^T
&= \begin{pmatrix}
1 & 1 & -1 & -5 \\
0 & 1 & 0 & 0 \\
0 & 0 & 1 & -1 \\
0 & -2 & 0 & 1
\end{pmatrix}.
\end{align*}
This matrix is upper triangular with order $v_1 < v_3 < v_4 < v_2$.
\end{example}


\pagebreak[3]

We next establish that certain subquivers in a mutation-acyclic quiver are acyclic.

\begin{theorem}
\label{thm:no vortex}
Suppose $Q$ is mutation-acyclic, and $v$ is a vertex in $Q$.
Then the subquiver of $Q$ supported by $\Out(v)$ (resp.\ $\In(v)$) is acyclic.
\end{theorem}

This theorem generalizes an example provided in \cite[Figure 1]{SevenAMatrices}. 

\begin{proof}
As $Q$ is mutation-acyclic, by Theorem~\ref{thm:SevenAdmissibleCompanions}, $Q$ has an admissible quasi-Cartan companion $A = (a_{pq})$ (Definition~\ref{def:admissible}).
Thus every subquiver of $Q$ also has an admissible quasi-Cartan companion. 
(By restricting the rows and columns of $A$.)

Thus we may reduce to the case that $v$ is a source (resp.\ sink), and $v$ is adjacent to all other vertices in $Q$.
If $\Out(v)$ (resp.\ $\In(v)$) contains an oriented cycle then it also contains an oriented chordless cycle (by a standard subdivision argument, any chord results in a smaller oriented cycle; cf.\ Lemma~\ref{lem:cycles restrict to chordless}).

Assume for contradiction that the undirected simple graph $K_Q$ consists only of a chordless forward-oriented cycle $\mathcal O = (w_0 - \cdots - w_\ell = w_0)$ and the source vertex $v$ with edges $w_i - v$ (the case with $v$ a sink is identical). 
Then an odd number of the $a_{w_i w_{i+1}}$ are positive.
In particular, at least one $a_{w_i w_{i+1}}$ is positive. 

Without loss of generality, say $a_{w_0 w_1} > 0$. 
Note that $A$ remains an admissible quasi-Cartan companion if we simultaneously toggle the signs of the row and column corresponding to $w_i$.
By repeatedly toggling signs, we may assume that $a_{w_i w_{i+1}} < 0$ for all $i>0$ [why?].
However, each of the chordless cycles $(w_i - w_{i+1} - v - w_i)$ in~$K_Q$ are not oriented, and so any admissible companion must satisfy an even number of the inequalities~${a_{w_i w_{i+1}}>0}, {a_{w_{i+1} v}>0},$ and~${a_{w_i v}>0}$.
In particular, one of $a_{w_0 v}$ or~$a_{w_{1} v}$ is positive, without loss of generality say $a_{w_{1} v} > 0$.
As~$a_{w_1 w_2} < 0$, we must have~$a_{w_2 v} > 0$ for the chordless cycle $(w_1 - w_{2} - v - w_1)$ to have an even number of positive signs, and by induction we find that $a_{w_i v} > 0$ for all $i$.
But then the chordless cycle $(w_0 - w_{1} - v - w_0)$ has three positive entries in~$A$, a contradiction.
\end{proof}

\begin{remark}
Suppose $Q$ is mutation-acyclic, and $v$ is a vertex in $Q$. 
Partition the other vertices of $Q$ into three sets $\Out(v), \In(v)$ and the set of vertices $S$ which are not adjacent to $v$. 
Consider the subquiver supported by $\Out(v) \cup S$ (resp.\ $\In(v) \cup S$).
Essentially the same argument as in Theorem~\ref{thm:no vortex} implies that every pair of vertices in $\Out(v)$ (resp.\ $\In(v)$) on a chordless oriented cycle in  $\Out(v) \cup S$ (resp.\ $\In(v) \cup S$) must be adjacent. 
Thus, for example, the quiver in Figure~\ref{fig:extended vortexish} is not mutation-acyclic.
\end{remark}

\begin{figure}[ht]
\includegraphics[alt={A six vertex quiver, consisting of an oriented five-cycle v1, v2, v3, v4, v5, v1 and a vertex v6 with arrows from v1 to v6 and v4 to v6. 
There is a single arrow between v2 and v3, v3 and v4, v4 and v5, v5 and v1, as well as three arrows from v1 to v2. Note that the number of arrows is irrelevant for the example; what matters is the orientation of the arrows and which vertices are in the inset or outset.}]{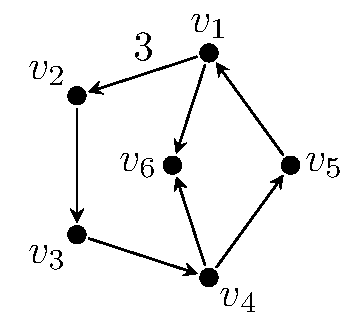}
\caption{A quiver on $6$ vertices which has no admissible quasi-Cartan companion, and so is not mutation-acyclic. The subquiver supported by $v_1, \ldots,v_5$ is an oriented cycle, contains two non-adjacent vertices in $\Out(v_6)$, but no vertices in $\In(v_6)$.}
\label{fig:extended vortexish}
\end{figure}


\pagebreak[2]

The following elementary lemma lets us convert a cycle in the underlying unoriented graph (Definition~\ref{def:underlying unoriented graph}) into a chordless cycle, while keeping a given vertex on an oriented path.

\begin{lemma}
\label{lem:cycles restrict to chordless}
Fix a vertex $v$ in a quiver $Q$.
Let $\mathcal{O}$ be a cycle in $K_Q$, which contains the path $u \rightarrow v \rightarrow w$ for some vertices $u,w$.
Then there exists a chordless cycle that is supported by a subset of the vertices of $\mathcal{O}$ and contains a path $ u' \rightarrow v \rightarrow w'$ for some vertices $u', w'$, not necessarily distinct from $u,w$.
\end{lemma}

\begin{proof}
We induct on the size of $\mathcal{O}$.
If $\mathcal{O}$ is chordless (i.e., if $\mathcal{O}$ only contains $u,v,$ and~$w$), then there is nothing to prove. 
Otherwise, $\mathcal{O}$ contains a chord.
If $\mathcal{O}$ contains a chord which does not involve $v$, then we get a shorter chordless cycle by including the chord edge, and removing vertices and edges so that $u,v,$ and $w$ remain. 
Suppose instead that $\mathcal{O}$ contains a chord with the vertex $v$, without loss of generality say~$v \rightarrow w'$ for a vertex $w'$ in $\mathcal O$ and distinct from $u,w$. 
We can replace the edge~$v \rightarrow w$ by~$v \rightarrow w'$ and remove vertices so that $u, v$ and $w'$ remain (and $w$ does not).
The claim follows by induction.
\end{proof}

Finally, we give a new way to check that two COQs are wiggle equivalent.

\begin{lemma}
\label{lem:wiggle equiv if both upper triangular}
Fix a quiver $Q$ and two linear orderings of its vertices $<,$ and $<'$.
Let~$U=(u_{vw})$ and $U'=(u'_{vw})$ be the unipotent companions of the linearly ordered quivers $(Q, <)$ and $(Q, <')$ respectively.
If $u_{vw} = u'_{vw}$ for every pair of vertices $v,w$, then the COQs associated to $(Q, <)$ and $(Q, <')$ are wiggle equivalent.
\end{lemma}

\begin{remark}
Note that the condition $u_{vw} = u'_{vw}$ for each pair of vertices is \emph{not} the same as requiring that the matrices $U$ and $U'$ are equal when they are written in their own respective linear orderings.
For example, the vertex $v$ in the first subscript may refer to a different row than in the second.
\end{remark}

\begin{proof}[Proof of Lemma~\ref{lem:wiggle equiv if both upper triangular}]
By Theorem~\ref{th:wiggle/winding}, it suffices to check that both COQs have the same winding numbers.
From \eqref{eq:wind formula}, it suffices to identify the arrows which ``wrap around'' the respective linear orderings (and, if they differ, their orientations).
For any two vertices $v,w$ we have $v < w$ if and only if $u_{vw} \neq 0$.
Thus the same arrows contribute regardless of if we use the linear ordering $<$ or $<'$, and so the winding numbers agree.
\end{proof}

\begin{example}
Let $Q$ be the acyclic quiver described in Figure~\ref{fig:cyclic orderings}.
Consider the two unipotent companions from different linear orderings of $Q$ below.
$$U= 
\begin{pmatrix} 1 & 3 &  2 & 1 \\ 0 & 1 & -1 & -1 \\ 0 & 0 & 1 & 0\\ 0 & 0 & 0 & 1 \end{pmatrix}
\quad 
U' = 
\begin{pmatrix} 1 & 3 &  1 & 2 \\ 0 & 1 & -1 & -1 \\ 0 & 0 & 1 & 0\\ 0 & 0 & 0 & 1 \end{pmatrix}$$
$$ \quad \quad v_4 < v_1 < v_2 < v_3  \quad \quad \quad \quad v_4 < v_1 < v_3 < v_2 $$
Each entry of $U$ matches $U'$; for example, the $2$ in the first row and third column of~$U$ corresponds to $u_{v_4, v_2}$ and agrees with the entry in the first row and fourth column of~$U'$. 
Thus $U$ and $U'$ are wiggle equivalent, indeed they are related by the wiggle~$(v_2 v_3)$.
\end{example}


\newpage

\section{{Proof of Theorem~\ref{thm:acyclic totally proper}}}
\label{sec:thm}

We begin by establishing and recalling notation for this section. 
Using the notation from Definition~\ref{def:indexing}, suppose our initial quiver $\Qinit$ is acyclic with (mutable) vertices $v_1, \ldots, v_n$. 
Fix a quiver $\tilde Q_t \in [ \widehat \Qinit ]$, with mutable part $\Qtmut$. 
Thus~$C_t$ denotes the $C$-matrix of $\Qtmut$, and $B_t$ denotes the $B$-matrix of $\Qtmut$.
Finally, recall from Definition~\ref{def:A t} the associated quasi-Cartan companions $A_{t_0}$ and $A_t = (a_{ij;t}) = C_t^T A_{t_0} C_t$.

\begin{definition}
\label{def:U t}
We associate to each $t \in \mathbb{T}_n$ an additional matrix 
$$U_t = (u_{vw;t}) = (A_t - B_t)/2.$$ 
Observe that $u_{vv;t}=1$ and $U_t^T - U_t = B_t$.
\end{definition}

\begin{example}
\label{eg:E6}
Let $\Qinit$ be the acyclic quiver of type $E_6$ shown in Figure~\ref{fig:E6}, and let $t \in \mathbb T_n$ be the vertex such that 
$$ t_0 \stackrel{3}{\text{---}} t_1 \stackrel{2}{\text{---}} t_2 \stackrel{5}{\text{---}} t_3 \stackrel{4}{\text{---}} t_4 \stackrel{2}{\text{---}} t_5 \stackrel{1}{\text{---}} t_6=t.$$ 

In this case,
$$ B_t = \begin{psmallmatrix}
0 & 0 & 0 & 1 & 0 & -1 \\
0 & 0 & -1 & -1 & 1 & 1 \\
0 & 1 & 0 & 0 & -1 & -1 \\
-1 & 1 & 0 & 0 & 0 & 0 \\
0 & -1 & 1 & 0 & 0 & 0 \\
1 & -1 & 1 & 0 & 0 & 0
\end{psmallmatrix}, \quad C_t = 
\begin{psmallmatrix}
-1 & 0 & 0 & 0 & 0 & 1 \\
0 & 0 & 0 & -1 & 0 & 1 \\
0 & -1 & 0 & 0 & 0 & 1 \\
0 & -1 & 1 & 0 & 0 & 0 \\
0 & 0 & 0 & 0 & -1 & 0 \\
0 & 0 & 0 & 0 & 0 & 1
\end{psmallmatrix},$$
and therefore 
$$A_t = \begin{psmallmatrix}
2 & 0 & 0 & -1 & 0 & -1 \\
0 & 2 & -1 & -1 & -1 & 1 \\
0 & -1 & 2 & 0 & 1 & -1 \\
-1 & -1 & 0 & 2 & 0 & 0 \\
0 & -1 & 1 & 0 & 2 & 0 \\
-1 & 1 & -1 & 0 & 0 & 2
\end{psmallmatrix}, \quad U_t = \begin{psmallmatrix}
1 & 0 & 0 & -1 & 0 & 0 \\
0 & 1 & 0 & 0 & -1 & 0 \\
0 & -1 & 1 & 0 & 1 & 0 \\
0 & -1 & 0 & 1 & 0 & 0 \\
0 & 0 & 0 & 0 & 1 & 0 \\
-1 & 1 & -1 & 0 & 0 & 1
\end{psmallmatrix}.$$
All of the above is written using the linear order $v_1 < v_2 < \cdots < v_6$.
\end{example} 

\begin{figure}[ht]

{ 

\newcommand\radius{0.9cm}
\includegraphics[alt={A tree quiver on six vertices v1, v2, v3, v4, v5, v6, with an arrow from v1 to v2, v2 to v3, v3 to v4, v4 to v5, and v3 to v6.}]{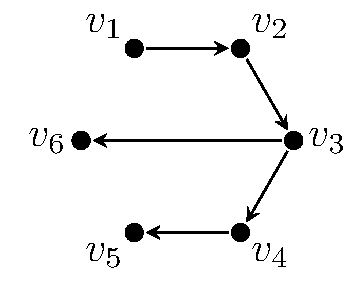}
\quad 
\includegraphics[alt={A tree quiver on six vertices v1, v2, v3, v4, v5, v6, with an arrow from v1 to v4, v4 to v2, v2 to v6, v6 to v1, v6 to v3, v3 to v2, v2 to v5, and v5 to v3.}]{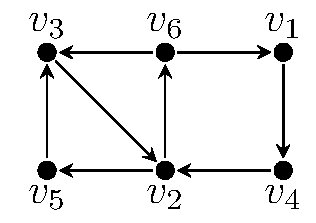}
}

$\Qinit$ \quad \quad \quad \quad \quad \quad \quad $\Qtmut$
\caption{The quiver $\Qinit$ of type $E_6$, and a mutation equivalent quiver $\Qtmut$.}
\label{fig:E6}
\end{figure}

We will construct a family of linear orderings such that $U_t$ is upper triangular when written using one of them (see Lemma~\ref{lem:U unipotent}). 
To facilitate the construction, we first create several new directed graphs from $\Qtmut$.

\begin{definition}
\label{def:tildeQ}
Fix a vertex $v_j$ of $\Qtmut$. 
Let $\Lgraph_{v_j}(\Qtmut)$ be the directed graph obtained by performing the following operations on~$\Qtmut$:
\begin{enumerate}
\item if $v_j$ is red (resp., green), then for each oriented $2$ path $v_i \rightarrow v_j \rightarrow v_k$ with $v_i,v_k$ both green (resp., red) in $\Qtmut$, add an arrow $v_k \stackrel{v_j}{\rightarrow} v_i$ labeled with the vertex $v_j$;
\item reverse all arrows from a red vertex directed toward a green vertex. Equivalently, by Theorem~\ref{thm:seven A properties}, reverse all arrows $v \rightarrow w$ with $a_{v w;t} > 0$. 
\end{enumerate}
\end{definition}

\begin{example}
\label{eg:E6v2}
Continuing with Example~\ref{eg:E6}, we see from the columns of $C_t$ that~$v_3, v_6$ are green while $v_1, v_2, v_4,$ and $v_5$ are red in $Q_t$.
Therefore each $\Lgraph_{v_i}(\Qtmut)$ will reverse the arrows $v_2 \rightarrow v_6$ and $v_5 \rightarrow v_3$ (agreeing with the two positive off-diagonal entries of~$A_t$), and we only add a new arrow to $\Lgraph_{v_i}(\Qtmut)$ if $i \in\{2,3,6\}$.
\end{example}

\usetikzlibrary{patterns}

\begin{figure}[ht]
{ 
\newcommand\hgap{3cm}
\newcommand\tildeI[1]{
	\filldraw[magenta] (1,1) circle (2pt);
	\filldraw[magenta] (0,0) circle (2pt);
	\filldraw[black!60!green] (-1,1) circle (2pt);
	\filldraw[magenta] (1,0) circle (2pt) ;
	\filldraw[magenta] (-1,0) circle (2pt);
	\filldraw[black!60!green] (0,1) circle (2pt);

	\filldraw[black] (1,1) circle (0pt) node[above] {$v_1$} coordinate (a);
	\filldraw[black] (0,0) circle (0pt) node[below] {$v_2$} coordinate (b);
	\filldraw[black] (-1,1) circle (0pt) node[above] {$v_3$} coordinate (c);
	\filldraw[black] (1,0) circle (0pt) node[below] {$v_4$} coordinate (d);
	\filldraw[black] (-1,0) circle (0pt) node[below] {$v_5$} coordinate (e);
	\filldraw[black] (0,1) circle (0pt) node[above] {$v_6$} coordinate (f);
	
	\draw[black, -{stealth}, shorten >=3pt, shorten <= 3pt] (a) -- (d);
	\draw[black, -{stealth}, shorten >=3pt, shorten <= 3pt] (d) -- (b); 
	\draw[black, -{stealth}, shorten >=3pt, shorten <= 3pt] (f) -- (b);
	\draw[black, -{stealth}, shorten >=3pt, shorten <= 3pt] (f) -- (a);
	\draw[black, -{stealth}, shorten >=3pt, shorten <= 3pt] (b) -- (e); 
	\draw[black, -{stealth}, shorten >=3pt, shorten <= 3pt] (c) -- (e);
	\draw[black, -{stealth}, shorten >=3pt, shorten <= 3pt] (c) -- (b);
	\draw[black, -{stealth}, shorten >=3pt, shorten <= 3pt] (f) -- (c);
	\draw (0,-0.3) node[below] {$\Lgraph_{v_{#1}}(\Qtmut)$};
}
\includegraphics[alt={The directed graphs associated to vertices v1, v2, and v3. In each, vertices v3 and v6 are colored green, while v1, v2, v4, and v5 are colored red. Each has six vertices v1, v2, v3, v4, v5, v6, and all have edges from v1 to v4, v4 to v2, v2 to v5, v6 to v3, v6 to v2, v6 to v1, v3 to v2, and v3 to v5. The graph associated to v1 has no other edges.}]{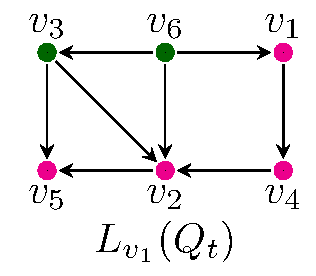}
\quad
\includegraphics[alt={The graph associated to v2 has an additional edge from v6 to v3 which is labeled v2.}]{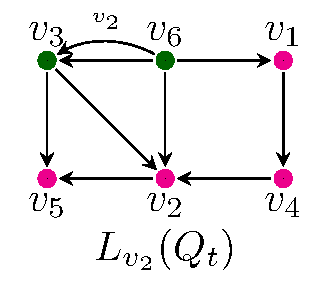}
\quad
\includegraphics[alt={ The graph associated to v3 has an additional edge from v2 to v5 labeled v3.}]{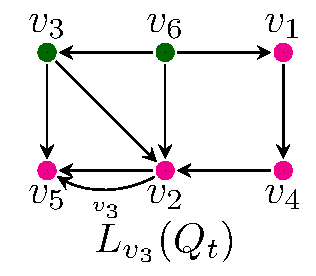}

\vspace{0.2cm}

\includegraphics[alt={The graph associated to v4 is identical to that of v1.}]{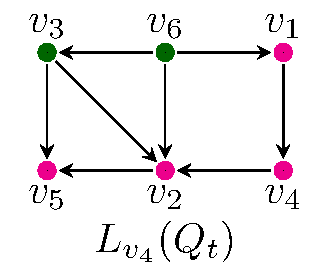}
\quad
\includegraphics[alt={The graph associated to v5 is identical to that of v1.}]{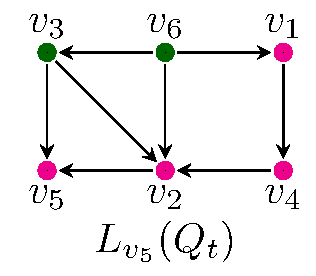}
\quad
\includegraphics[alt={The graph associated to v6 is identical to that of v1.}]{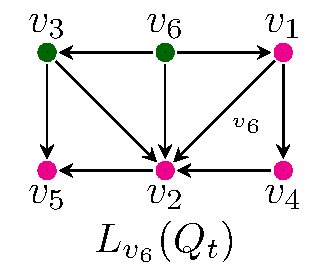}
}
\caption{The six directed graphs $\Lgraph_{v_i}(\Qtmut)$ associated to $\Qtmut$ in Example~\ref{eg:E6}. 
Vertices~$v_3, v_6$ are green, while vertices $v_1, v_2, v_4, v_5$ are red.} 
\label{fig:Qtildes}
\end{figure}
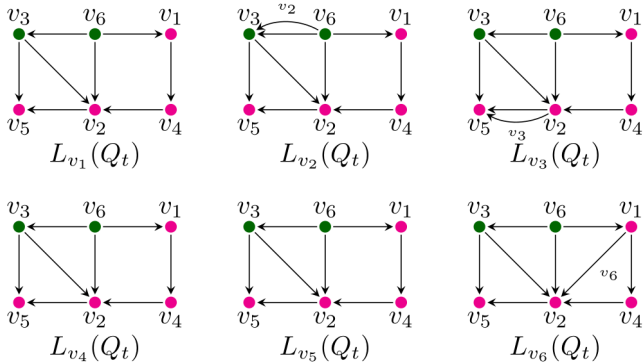

\begin{lemma}
\label{lem:Q tilde acyclic}
Each directed graph $\Lgraph_{v}(\Qtmut)$ is acyclic. (In particular, there are no oriented $2$-cycles; $\Lgraph_{v}(\Qtmut)$ is an acyclic quiver with some marked arrows.)
Furthermore, the subquiver of $\Qtmut$ supported by all red (resp., green) vertices is acyclic.
\end{lemma}


\begin{proof}
By construction, there are no directed paths from a red vertex to a green vertex in $\Lgraph_{v}(\Qtmut)$.
So any oriented cycle in $\Lgraph_v(\Qtmut)$ must have all vertices of the same color (red or green). 

We first consider oriented cycles which do not use any arrows labeled $v$, i.e.\ oriented cycles in $\Qtmut$ which all have the same color.
Let $\boxQ_t$ be constructed from $\tilde Q_t$ as in Definition~\ref{def:boxQ}. 
By construction of $\boxQ_t$, the set $\In(\mathring{v})$ (resp., $\Out(\mathring{v})$) is precisely the red (resp., green) vertices in $\Qtmut$.
As $\boxQ_0$ is acyclic, by Corollary~\ref{cor:Cmats and gluing}, $\boxQ_t$ is mutation-acyclic. 
By Theorem~\ref{thm:no vortex}, $\In(\mathring{v})$ (resp., $\Out(\mathring{v})$) is an acyclic subquiver of $\boxQ_t$.
Thus there are no oriented cycles of all red (resp., green) vertices in $\Qtmut$.

Without loss of generality, suppose $v$ is red. 
The case where $v$ is green is similar.
Then we do not add any arrows labeled $v$ between red 
vertices, so we have shown there are no oriented cycles of red 
vertices.
Further, any directed cycle involving green 
vertices must use at least one arrow labeled $v$ (those created in Step~$1$ above).
As all arrows labeled $v$ in $\Lgraph_{v}(\Qtmut)$ are oriented from $\Out(v)$ towards $\In(v)$, any oriented cycle must also involve at least one arrow which is not labeled $v$.
We will now argue that any other cycles in $\Lgraph_{v}(\Qtmut)$ imply a problematic cycle $\mathcal{O}$ in the underlying unoriented simple graph $K_{\Qtmut}$ of $\Qtmut$. 

Suppose there is an oriented cycle $w_0 \stackrel{v}{\rightarrow} w_1 \rightarrow \cdots \rightarrow w_k = w_0$ in $\Lgraph_v(\Qtmut)$ that involves just one arrow $w_0 \stackrel{v}{\rightarrow} w_1$ labeled $v$.
By construction of $\Lgraph_v(\Qtmut)$, there is a path~${w_0 \leftarrow v \leftarrow w_1}$ in $\Qtmut$. 
Thus there is a cycle ${\mathcal O = (w_0 - v - w_1 - w_2- \cdots - w_k)}$ in~$K_{\Qtmut}$.
If~$\mathcal{O}$ is not chordless then replace it with a chordless cycle containing $v$ by Lemma~\ref{lem:cycles restrict to chordless}. 
Because the subquiver of green vertices in~$\Qtmut$ is acyclic, any chord between the green vertices~$w_i, w_j$ has the same orientation as the directed path it replaced. 
Therefore~$\mathcal{O}$ is a chordless cycle which is not oriented. 
But the only edge on $\mathcal O$ which has a positive entry in~$A_t$ is $w_0 - v$, contradicting Theorem~\ref{thm:seven A properties}.
So there are no cycles in $\Lgraph_v(\Qtmut)$ with a single arrow labeled $v$.

Suppose there is an oriented cycle in $\Lgraph_v(\Qtmut)$ with at least $2$ arrows labeled $v$, say
$$w_0 \stackrel{v}{\rightarrow} w_1 \rightarrow \cdots \rightarrow w_j \stackrel{v}{\rightarrow} w_{j+1} \rightarrow \cdots \rightarrow w_k = w_0$$ 
with $j>1$ and none of the arrows on the path $w_1 \rightarrow w_2 \rightarrow \cdots \rightarrow w_j$ labeled $v$ (the other arrows may or may not be labeled $v$).
Thus $\mathcal O = (v - w_1 - \cdots - w_j - v)$ is a cycle in~$K_{\Qtmut}$ which is {not} oriented.
After applying Lemma~\ref{lem:cycles restrict to chordless} if necessary, we may assume~$\mathcal O$ is a chordless cycle.
But the only edge on $\mathcal O$ which has a positive entry in~$A_t$ is $w_j - v$, contradicting Theorem~\ref{thm:seven A properties}. 
Thus $\Lgraph_v(\Qtmut)$ is acyclic.
\end{proof}

\begin{definition}
\label{def:order from tilde}
A \emph{linear extension} of $\Lgraph_{v}(\Qtmut)$ is a linear order $<$ on the vertices such that $v_i < v_j$ whenever $v_i  \rightarrow v_j$ (it does not matter whether the arrow is labeled~$v$). 
Let $\sigma_{v}$ be a cyclic ordering associated to a linear extension of $\Lgraph_{v}(\Qtmut)$.
\end{definition}

\begin{example}
\label{eg:E6v3}
Continuing with Examples~\ref{eg:E6} and~\ref{eg:E6v2}, 
there are several possible choices for the cyclic orderings $\sigma_{v}$ for each fixed $v$, though all of them result in wiggle equivalent COQs $(\Qtmut, \sigma_v)$. 
We could take $$\sigma_{v_1} = \sigma_{v_2} = \sigma_{v_3} = (v_6, v_1, v_4, v_3, v_2, v_5), \text{ and}$$
$$\sigma_{v_4} = \sigma_{v_5} = \sigma_{v_6} = (v_6, v_3, v_1, v_4, v_2, v_5).$$
\end{example}

\begin{lemma}
\label{lem:U unipotent}
The matrix $U_t=(u_{v_i v_j;t})$ is upper triangular when written with a linear extension $<$ of $\Lgraph_{v}(\Qtmut)$. Thus $U_t$ is a unipotent companion of $(\Qtmut, \sigma_v)$.
\end{lemma}

\begin{proof}
Suppose $v_i,v_j$ are vertices such that $u_{v_i v_j;t} \neq 0$. 
We want to show that $v_i < v_j$.
Because $|a_{v_i v_j;t}| = |b_{v_i v_j;t}|$, and $u_{v_i v_j;t} = \frac{a_{v_i v_j;t} - b_{v_i v_j;t}}{2} \neq 0$, we have $a_{v_i v_j;t} = -b_{v_i v_j;t}$.

If $a_{v_i v_j;t}>0$ then $v_i \leftarrow v_j$ (as $b_{v_i v_j;t} < 0$).
By Definition~\ref{def:tildeQ}, the graph $\Lgraph_{v}(\Qtmut)$ has an arrow~${v_i \rightarrow v_j}$. 
Thus $v_i < v_j$ in this case.

Otherwise $a_{v_i v_j;t}<0$ and $v_i \rightarrow v_j$ in~$\Lgraph_{v}(\Qtmut)$, so we have $v_i < v_j$.
\end{proof}

\begin{example}
\label{eg:E6v4}
Continuing with Example~\ref{eg:E6}, we find that $U_t$ is a unipotent companion of $(\Qtmut, \sigma_{v})$, for any vertex $v$. 
As suggested by Example~\ref{eg:E6v3}, if we use the linear order $v_6 < v_3 < v_1< v_4< v_2< v_5$, then
$$ U_t = \begin{pmatrix}
1 & -1 & -1 & 0 & 1 & 0 \\
0 & 1 & 0 & 0 & -1 & 1 \\
0 & 0 & 1 & -1 & 0 & 0 \\
0 & 0 & 0 & 1 & -1 & 0 \\
0 & 0 & 0 & 0 & 1 & -1 \\
0 & 0 & 0 & 0 & 0 & 1
\end{pmatrix}.$$
\end{example}

\pagebreak[2]

\begin{lemma}
\label{lem: Q tilde wiggle}
The COQs $(\Qtmut, \sigma_v), (\Qtmut, \sigma_w)$ are wiggle equivalent for all vertices~$v,w$.
\end{lemma}
\begin{proof}
By Lemma~\ref{lem:U unipotent}, the matrix $U_t$ is a unipotent companion of both $(\Qtmut, \sigma_v)$ and~$(\Qtmut, \sigma_w)$. 
The claim follows from Lemma~\ref{lem:wiggle equiv if both upper triangular}.
\end{proof}


\begin{lemma}
\label{lem:U gives proper}
The COQ $(\Qtmut, \sigma_v)$ is proper.
\end{lemma}

\begin{proof}
It suffices to check that each vertex $q$ is proper in the COQ~$(\Qtmut, \sigma_q)$ by Lemma~\ref{lem: Q tilde wiggle}. 
Let $<$ be a linear extension of  $\Lgraph_{q}(\Qtmut)$.
We consider all oriented paths~${p \rightarrow q \rightarrow r}$ in cases, based on the color (green or red) of $p,q,$ and $r$ in $\Qtmut$.
As noted in Lemma~\ref{lem:Q tilde acyclic}, the subquiver of all red or all green vertices is acyclic.
We have chosen our order $<$ to be compatible with these acyclic subquivers.
Thus if all three of $p,q,r$ are green (resp., red) then $p \rightarrow q \rightarrow r$ is a right turn at $q$. 
A similar argument applies if $p$ or $r$ is a different color than the other two vertices.

This leaves us with the alternating patterns: red-green-red or green-red-green. 
By construction of $<$, if $p,r$ are both a different color than $q$, then there is an arrow $r \stackrel{q}{\rightarrow} p$ in  $\Lgraph_{q}(\Qtmut)$, and so $r<p$. 
Thus the path $p \rightarrow q \rightarrow r$ is again a right turn at $q$.
\end{proof}

\begin{example}
Continuing with Examples~\ref{eg:E6} and~\ref{eg:E6v3}, Figure~\ref{fig:E6COQ} shows the COQ $(\Qtmut, \sigma_{v_1})$. 
Note that every oriented path of length $2$ is a right turn.
\end{example}

\begin{figure}[ht]

{ 

\newcommand\radius{1.2cm}
\includegraphics[alt={A COQ whose quiver is as in the previous examples, and whose cyclic ordering is (v1, v4, v3, v2, v5, v6). Every vertex is proper, for example the oriented 2-path v2 to v6 to v1 is a right turn as v2 v6 v1 appear in clockwise order.}, scale=1.2]{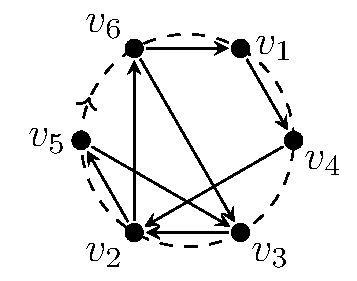}
} 
\caption{The proper COQ $(\Qtmut, \sigma_{v_1})$.}
\label{fig:E6COQ}
\end{figure}

\begin{proof}[Proof of Theorem~\ref{thm:acyclic totally proper}]
Let $\sigma_t = \sigma_{v_1}$ be a cyclic ordering constructed from $\Qtmut$ as in Definition~\ref{def:order from tilde}.
By Lemma~\ref{lem:U unipotent}, $U_t$ is a unipotent companion of the COQ $(\Qtmut, \sigma_t)$.
By Lemma~\ref{lem:U gives proper}, $(\Qtmut,\sigma_t)$ is proper.
It remains to show that these COQs (as $t \in \mathbb T_n$ varies) are all proper mutation equivalent.

Fix an arbitrary $t' \stackrel{i}{\text{---}} t$. 
It suffices to show that the COQ $(\mu_{v_i}(\Qtmut), \sigma_{t'})$ is in the wiggle equivalence class $\mu_{v_i}((\Qtmut, \sigma_t))$. Let $<$ be a linear extension of $\Lgraph_{v_i}(\Qtmut)$.
Say that $v_i$ is green (resp., red) in $\Qtmut$.
For every vertex $w \in \Out(v_i)$ (resp., $\In(v_i)$), we have $v_i < w$ (resp. $w < v_i$).
By Proposition~\ref{prop:A matrix mutation acyclic}, we have
$$B_{t'} = M_{Q_t}(v_i, \varepsilon) B_t M_{Q_t}(v_i, \varepsilon)^T,$$
$$A_{t'} = M_{Q_t}(v_i, \varepsilon) A_t M_{Q_t}(v_i, \varepsilon)^T$$ 
for $\varepsilon = 1$ (resp., $-1$).
Thus 
\begin{align*}
U_{t'} &=  \frac 1 2 \big ( M_{Q_t}(v_i, \varepsilon) A_t M_{Q_t}(v_i, \varepsilon)^T - M_{Q_t}(v_i, \varepsilon) B_t M_{Q_t}(v_i, \varepsilon)^T \big ) \\
&= M_{Q_t}(v_i, \varepsilon) U_t M_{Q_t}(v_i, \varepsilon)^T.
\end{align*}

In particular, $U_{t'}$ is exactly the unipotent companion of $\mu_{v_i}((\Qtmut, \sigma_t))$ produced by Lemma~\ref{lem:mutates like B v2gen} (resp. \ref{lem:mutates like B v1gen}).
The theorem then follows from Lemma~\ref{lem:wiggle equiv if both upper triangular}.
\end{proof}

\newpage 
\section{Applications}
\label{sec:cor}

As noted in Theorem~\ref{thm:payoff}, the integral congruence class of a unipotent companion is a proper mutation invariant of COQs. 
If a COQ is totally proper, then every mutation is proper and so we have a mutation invariant of the underlying quiver. 

\begin{theorem}[{\cite[Theorem~14.3, Algorithm~14.12]{COQ}}]
\label{thm:uniq}
A quiver $Q$ has at most one totally proper cyclic ordering (up to wiggle equivalence).
There is an efficient algorithm that computes a cyclic ordering $\sigma_Q$ such that if $Q$ has a totally proper cyclic ordering, then $(Q, \sigma_Q)$ is totally proper.
\end{theorem}

If a quiver $Q$ is mutation equivalent to a totally proper quiver $Q'$, then the integral congruence class of a unipotent companion of the COQ $(Q, \sigma_Q)$ will agree with that of $(Q', \sigma_{Q'})$. 
In particular, their Alexander polynomials and Markov invariants (Definition~\ref{def:alexander and markov}) will agree.

By Theorem~\ref{thm:acyclic totally proper}, we now have many new totally proper quivers. 
Let us discuss a few examples.

\begin{corollary}
\label{cor:acyclic cycles}
Suppose $Q$ is an acyclic quiver whose underlying unoriented simple graph is a chordless cycle $\mathcal O$.
Then a cyclic ordering $\sigma$ of the vertices of $Q$ is totally proper if and only if $\wind_\sigma(\mathcal O) = 0$.
\end{corollary}

Note that Corollary~\ref{cor:acyclic cycles} makes no assumption about the multiplicities of the arrows.

\begin{example}
\label{eg:acyclic cycles}
As an illustration, we compute some mutation invariants for a few families.
Let $\tilde A(r, \ell)$ denote a quiver with $n = r + \ell$ vertices $v_1, \ldots, v_n$ which has $r$ locations with a single arrow $v_i \rightarrow v_{i+1}$ and $\ell$ locations with a single arrow $v_i \leftarrow v_{i+1}$ (see \cite[Example~6.9]{cats1}).
While there are many ways to place the arrows, all are related by sink and/or source mutations.
Thus we may assume that $\tilde A(r, \ell)$ has arrows
$$v_1 \rightarrow v_2 \rightarrow \cdots \rightarrow v_r \rightarrow v_{r+1} \leftarrow \cdots \leftarrow v_{n} \leftarrow v_1.$$
For $r, \ell > 0$, the quiver $\tilde A(r, \ell)$ is acyclic, and thus totally proper.
One totally proper cyclic ordering is $\sigma_{r\ell} = (v_1, \ldots, v_{r-1}, v_n, v_{n-1}, \ldots, v_{r+1}, v_r)$.

The Alexander polynomial of~$(\tilde A(r, \ell), \sigma_{r\ell})$ is 
$$\Delta_{\tilde A(r, \ell)}(t)= (t^\ell - (-1)^{\ell}) (t^r - (-1)^r).$$
Every quiver mutation equivalent to $(\tilde A(r, \ell), \sigma_{r \ell})$ has integrally congruent unipotent companions, and in particular a matching Alexander polynomial.

Thus, if $r=3$ and $\ell=2$ then one unipotent companion is
$$U = \begin{pmatrix}
1 & -1 & 0 & -1 & 0 \\
0 & 1 & -1 & 0 & 0 \\
0 & 0 & 1 & 0 & -1 \\
0 & 0 & 0 & 1 & -1 \\
0 & 0 & 0 & 0 & 1
\end{pmatrix},$$
and we can compute the Alexander polynomial $$\Delta_{\tilde A(3,2)}(t) = \det(t U - U^T)=t^5 - t^3 + t^2 - 1.$$
\hide{
We compute unipotent companions and their associated Alexander polynomials:

\begin{itemize}
\item  $(\tilde A(3, 1), \sigma_{3,1})$: One unipotent companion is
$$U = \begin{pmatrix}
1 & -1 & 0 & -1 \\
0 & 1 & -1 & 0 \\
0 & 0 & 1 & -1 \\
0 & 0 & 0 & 1 
\end{pmatrix},$$
the associated Alexander polynomial is $\Delta_{\tilde A(3,1)}(t) =\det(t U - U^T) = t^4 + t^3 + t + 1$.
\item  $(\tilde A(2, 2), \sigma_{2,2})$: One unipotent companion is
$$\begin{pmatrix}
1 & -1 & -1 & 0 \\
0 & 1 & 0 & -1 \\
0 & 0 & 1 & -1 \\
0 & 0 & 0 & 1 
\end{pmatrix},$$
the associated Alexander polynomial is $\Delta_{\tilde A(2,2)}(t) =t^4 - 2 t^2 + 1$. 
\item  $(\tilde A(4, 1), \sigma_{4,1})$: One unipotent companion is
$$\begin{pmatrix}
1 & -1 & 0 & 0 & -1 \\
0 & 1 & -1 & 0 & 0 \\
0 & 0 & 1 & -1 & 0 \\
0 & 0 & 0 & 1 & -1 \\
0 & 0 & 0 & 0 & 1
\end{pmatrix},$$
the associated Alexander polynomial is $\Delta_{\tilde A(4,1)}(t) =t^5 + t^4 - t - 1$.
\item  $(\tilde A(3, 2), \sigma_{3,2})$: One unipotent companion is
$$\begin{pmatrix}
1 & -1 & 0 & -1 & 0 \\
0 & 1 & -1 & 0 & 0 \\
0 & 0 & 1 & 0 & -1 \\
0 & 0 & 0 & 1 & -1 \\
0 & 0 & 0 & 0 & 1
\end{pmatrix},$$
the associated Alexander polynomial is $\Delta_{\tilde A(3,2)}(t) =t^5 - t^3 + t^2 - 1$.
\end{itemize}
}
\end{example}

\begin{example}
\label{eg:cycles with multiplicities}
We can expand Example~\ref{eg:acyclic cycles} by considering quivers whose simple undirected graphs are cycles but have multiple arrows between vertices.
Choose nonnegative integers $r, \ell$ and positive integers $a_1, \ldots, a_{r+\ell}$.
Let $C(r, \ell)$ denote the quiver with $n = r + \ell$ vertices $v_i$, and arrows
$$v_1 \stackrel{a_1}\rightarrow v_2 \stackrel{a_2}\rightarrow \cdots \stackrel{a_{r-1}}\rightarrow v_r \stackrel{a_r}\rightarrow v_{r+1} \stackrel{a_{r+1}}\leftarrow \cdots \stackrel{a_{n-1}}\leftarrow v_{n} \stackrel{a_n}\leftarrow v_1.$$
For $r,\ell > 0$, the cyclic ordering $\sigma_{r \ell}$ is totally proper.

We compute the Alexander polynomials for several small values of $r, \ell$. 
Recall that for a $4$-vertex quiver $Q$ we can write the Alexander polynomial as 
$$\Delta_Q(t) = (t-1)^4 + M_Q \cdot t(t-1)^2 + \det(B_Q) \cdot t^2.$$
In particular,
\begin{align*}
\Delta_{C(4,0)}(t) =& (t-1)^4 + (a_1^2 + a_2^2 + a_3^2 + a_4^2-a_1a_2a_3a_4) t(t-1)^2 \\ 
& + ( a_1^2a_3^2 + a_2^2a_4^2 - 2 a_1 a_2 a_3 a_4) t^2; \\
\Delta_{C(3,1)}(t) =& (t-1)^4 + ( a_1^2 + a_2^2 + a_3^2 + a_4^2 + a_1a_2a_3a_4 ) t (t-1)^2 \\ 
& + ( a_1^2a_3^2 + a_2^2a_4^2 + 2 a_1 a_2 a_3 a_4 ) t^2; \\
\Delta_{C(2,2)}(t) =& (t-1)^4 + ( a_1^2 + a_2^2 + a_3^2 + a_4^2 ) t(t-1)^2 \\ 
& + ( a_1^2a_3^2 + a_2^2a_4^2 - 2a_1a_2a_3a_4 ) t^2.
\end{align*}

\noindent
For a $5$-vertex quiver $Q$, we can write the Alexander polynomial as
$$\Delta_Q(t) = (t-1)^5 + M_Q \cdot t(t-1)^3 + d \cdot t^2 (t-1)$$
for some integer coefficient $d$. In particular,
\begin{align*}
\Delta_{C(4,1)}(t) =& (t-1)^5 + (a_1^2 + a_2^2 + a_3^2 + a_4^2 + a_5^2 + a_1 a_2 a_3 a_4 a_5) t (t-1)^3 \\ 
& + ( a_1^2 a_3^2 + a_1^2 a_4^2 + a_2^2 a_4^2 + a_2^2 a_5^2 + a_3^2 a_5^2 - 3 a_1a_2a_3a_4a_5) t^2(t - 1); \\
\Delta_{C(3,2)}(t) =& (t-1)^5 + ( a_1^2 + a_2^2 + a_3^2 + a_4^2 + a_5^2) t (t-1)^3 \\ 
& + (a_1^2a_3^2 + a_1^2a_4^2 + a_2^2a_4^2 + a_2^2a_5^2 + a_3^2a_5^2 - a_1a_2a_3a_4a_5)  t^2(t - 1).
\end{align*}

These examples suggest a pattern for the Markov invariant, agreeing with Proposition~\ref{prop:markov positive}. Without loss of generality, assume $r > 0$. Then
\begin{equation*}
M_{C(r, \ell)} = 
\begin{cases} 
\sum_i^n a_i^2 & \text{if } \ell \geq 2; \\ 
\sum_i^n a_i^2 + \prod_i^n a_i  & \text{if } \ell =1; \\ 
\sum_i^n a_i^2 - \prod_i^n a_i & \text{if } \ell =0.
\end{cases}
\end{equation*}
\end{example}

\begin{remark}
\label{rem:impossible values}
Recall Lagrange's four squares theorem, which states that every nonnegative integer is the sum of at most four squares (of integers).
Thus there is a totally proper $4$-vertex COQ with Markov invariant $x$ for every $x \geq 0$.
However some integers are not the sum of exactly four \emph{positive} squares. 
For example, $2^{2k+1}$, or $29$ (see \cite{A000534}).
Thus any quiver whose associated Markov invariant is one of these values must not be mutation equivalent to $C(2, 2)$ (for any positive values $a_i$).
\end{remark}

\hide{ 
\begin{figure}[h]

{
\newcommand\radii{1cm}
\includegraphics[alt={A 4 vertex quiver with vertices v1, v2, v3, v4. The only arrows in the quiver lie in the oriented three-cycles. There are 4 arrows from v1 to v2, 2 arrows from v2 to v4, and 22 arrows from v4 to v1. There are 5 arrows from v1 to v3, and 3 arrows from v3 to v4. With the 22 arrows from v4 to v1, this makes the second oriented three-cycle.}]{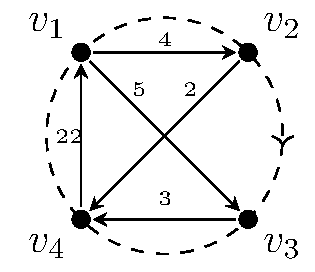}
}

\caption{Four vertex COQs which are not mutation equivalent to any $C(2,2)$.}
\label{fig:not mu cycles}

\end{figure}
}

We have the following corollaries of Proposition~\ref{prop:markov positive}.

\begin{corollary}
\label{cor:acyclic positive}
If $Q$ is a totally proper mutation-acyclic COQ on~$n$ vertices then $M_Q \geq 0$.
If, additionally, $Q$ is connected 
then the Markov invariant $M_Q \geq n-1$. 
\end{corollary}

\begin{corollary}
\label{cor:markov finite}
For a fixed integer $k$, there are only finitely many acyclic quivers $Q$ with $M_Q=k$ (and hence only finitely many acyclic quivers with the same \emph{Alexander polynomial}). 
In particular, if $M_Q = n-1$ and $Q$ is connected then $Q$ is mutation equivalent to a quiver whose underlying undirected graph is a tree (with no parallel arrows).
\end{corollary}

Note that $M_Q$ is monotone in all of the entries of $B_Q$ for an acyclic quiver $Q$. 
Therefore we can enumerate the acyclic quivers with $M_Q = k$ by searching all acyclic quivers with weights at most, say, $\sqrt{k}$ for matching Markov invariants.

\begin{corollary}
\label{cor:acyclic det inequality}
Let $Q$ be a totally proper and mutation-acyclic COQ on $n$ vertices. 
Then the Markov invariant and the determinant of the exchange matrix satisfy the following inequality:
$$\det(B_Q) \leq (2 M_Q)^{n/2}.$$
\end{corollary}

\begin{proof}
It suffices to consider acyclic COQs $Q$ with vertices $v_1, \ldots, v_n$.
By Proposition~\ref{prop:markov positive}, $\sum_{i<j} b_{v_i v_j}^2 \leq M_Q.$
Recall the \emph{Frobenius norm} $||B_Q||_F = \big (\sum_{i,j} b_{v_i v_j}^2\big )^{1/2}$ of~$B_Q$ (see also the Hilbert-Schmidt norm or Schur norm), thus
$||B_Q||_F^2 \leq 2 M_Q.$

As $\det(B_Q) = 0$ if $n$ is odd, we assume $n$ is even.
As $B_Q$ is skew-symmetric, its eigenvalues are all imaginary and come in signed pairs, say they have norms~$|\lambda_1|, \ldots, |\lambda_{n/2}|$.
Let $\lambda = \max( |\lambda_j|)$. 
Then clearly $\det(B_Q) \leq \lambda^n$.
As $B_Q$ is skew-symmetric the spectral norm $\rho(B_Q) = \lambda$.
It is a classical fact that the spectral norm is a lower bound on the Frobenius norm: $\lambda \leq ||B_Q||_F$. 
Therefore we have
\begin{equation*}
\det(B_Q) \leq \lambda^n \leq ||B_Q||_F^{n} \leq (2 M_Q)^{n/2}.\qedhere
\end{equation*}
\end{proof}

Corollaries~\ref{cor:acyclic positive},~\ref{cor:markov finite} and~\ref{cor:acyclic det inequality} give short proofs that certain quivers are not mutation-acyclic.

\begin{example}
\label{eg:somos4}
The \emph{Somos sequences} are integer sequences, the first $5$ of which are associated to cluster algebras with particularly symmetric quivers.
The Somos-$4$ quiver $S$ is shown in Figure~\ref{fig:somos4}.
The distinguished cyclic ordering $\sigma_S = (v_1, v_4, v_2, v_3)$ (see \cite{COQ}[Remark 11.11]) is the only potentially totally proper ordering for $S$.
This quiver is not mutation-acyclic though, which is quickly determined by computing the Markov invariant of the COQ $(S, \sigma_S)$ and applying Corollary~\ref{cor:acyclic positive}:
$$ M_S = 1^2 + 1^2 + 2^2 + 3^2 + 2^2 + 1^2 - 6 - 6 + 2 + 2 - 12 = 0.$$
\end{example}

\begin{figure}[ht]

\includegraphics[alt={A four vertex quiver with vertices v1, v2, v3, v4 and arrows as follows: One arrow from v1 to v2, v3 to v4, and v1 to v4; two arrows from v3 to v1 and v4 to v2; three arrows from v2 to v3. Mutation at v1 gives an isomorphic quiver.}]{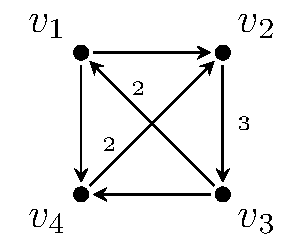}

\caption{The Somos-$4$ quiver $S$.}
\label{fig:somos4}
\end{figure}

\begin{remark}
Fix an integer $a \geq 1$. 
Example~\ref{eg:somos4} generalizes to the $4$-vertex quiver $aS$, which has the same vertices as $S$, and $a$ arrows for every arrow in $S$ (thus, for example, there are $3a$ arrows from $v_2$ to $v_3$).
Each has a nonpositive Markov invariant.
\end{remark}

\begin{example}
\label{eg:4cycle no acyclics}
Continuing with Example~\ref{eg:cycles with multiplicities}, consider the COQ $C(4,0)$ with cyclic ordering $(v_1, v_2, v_3, v_4)$, shown below. 
{
\newcommand\cyclicVerts[5]{%
	\filldraw[black] (#4,#5)++(135:0.8cm) circle (2pt) node[above left=-1pt] {$v_1$} coordinate (v_1);
	\filldraw[black] (#4,#5)++(45:0.8cm) circle (2pt) node[above right=-1pt] {$#1$} coordinate (#1);
	\filldraw[black] (#4,#5)++(-45:0.8cm) circle (2pt) node[below right=-1pt] {$#2$} coordinate (#2);
	\filldraw[black] (#4,#5)++(-135:0.8cm) circle (2pt) node[below left=-1pt] {$#3$} coordinate (#3);
	\draw[black, dashed, decoration={markings, mark=at position 0 with {\arrow{<}}}, postaction={decorate}] (#4,#5) circle (0.8cm);
}
$$\includegraphics[alt={A COQ with cyclic ordering (v1,v2,v3,v4) and a(i) arrows from v(i) to v(i+1) (indicies mod 4).}]{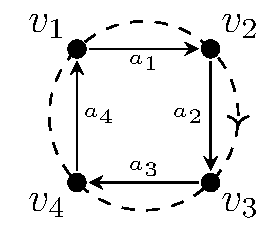}$$}

Some such quivers are mutation-acyclic, for example when $a_i=1$ for all $i$ then $C(4,0)$ has type $D_4$.
Others are not. 
Warkentin showed in \cite[Example 6.11]{Warkentin} that if $a_1=a_3\geq 2$ and $a_2=a_4\geq 2$ then this quiver is not mutation-acyclic (by a precise description of the mutation class).
The same can be observed by Corollary~\ref{cor:acyclic positive}, noting that~$M_{C(4,0)} \leq 0$ for such quivers.

Suppose that $a_4 = a_1 a_2 a_3$. Then $M_{C(4,0)} = a_1^2 + a_2^2 + a_3^2 \geq 3.$
Thus we cannot determine whether the quiver is mutation-acylic or not from Corollary~\ref{cor:acyclic positive}.
However, $$\det(B_{C(4,0)}) = (a_1 a_3 + a_1 a_2^2 a_3)^2 = (a_2^2 + 1)^2 (a_1 a_3)^2.$$
So we have $\det(B_{C(4,0)}) > ( 2 M_{C(4,0)})^2$ whenever (for example) $a_2 \geq \max(a_1, a_3)$ and $a_1 a_3 \geq 6$. 
Thus, for these values of $a_i$, the quiver $C(4,0)$ is again not mutation-acyclic by Corollary~\ref{cor:acyclic det inequality}.
\end{example}

\begin{example}
\label{ATP:eg:all M4=6}
As an application of Corollary~\ref{cor:markov finite}, we enumerate all of the acyclic $4$-vertex quivers with $M_Q \leq 6$ (up to isomorphism and mutation equivalence).
Because $M_Q \geq \sum_{i<j} b_{ij}^2$, we may assume $b_{ij} \leq 2.$
If $M_Q<3$ then $Q$ must be disconnected. 
There is only one quiver with $0$ or $1$ arrows.
There are two quivers with $M_Q=2$, corresponding to type $A_3$ (with an isolated vertex), and the union of two copies of $A_2$.
Up to mutation equivalence, there are two tree quivers on $4$-vertecies, corresponding to types $A_4, D_4$, both of which have $M_Q = 3$.
The only mutation class of connected acyclic quivers with $M_Q=4$ is $\tilde A(2,2)$. 
There are also the disconnected quivers corresponding to $v_1 \stackrel{2}{\rightarrow} v_2$ and $\tilde A(2,1)$ (with all other vertices isolated).
There are two mutation-classes of connected acyclic quivers $Q$ with $M_Q=5$, as well as two classes of disconnected quivers. 
Representitives are shown below:
{
\newcommand\cyclicVerts[5]{%
	\draw[black, dashed, decoration={markings, mark=at position 0 with {\arrow{<}}}, postaction={decorate}] (#4,#5) circle (0.8cm);
	\filldraw[black] (#4,#5)++(135:0.8cm) circle (2pt) node[above left=-1pt] {$v_1$} coordinate (v_1);
	\filldraw[black] (#4,#5)++(45:0.8cm) circle (2pt) node[above right=-1pt] {$#1$} coordinate (#1);
	\filldraw[black] (#4,#5)++(-45:0.8cm) circle (2pt) node[below right=-1pt] {$#2$} coordinate (#2);
	\filldraw[black] (#4,#5)++(-135:0.8cm) circle (2pt) node[below left=-1pt] {$#3$} coordinate (#3);
}
$$\includegraphics[alt={Four COQs, each with cyclic ordering (v1, v2, v3, v4). The leftmost has arrows from v1 to v2, v2 to v3, v3 to v4, and v1 to v4. The next has an arrow from v1 to v2, v2 to v3, v3 to v4, and v1 to v3. The next has 2 arrows from v1 to v2, and one arrow from v3 to v4. The rightmost has 2 arrows from v1 to v2, and one arrow from v3 to v4.}]{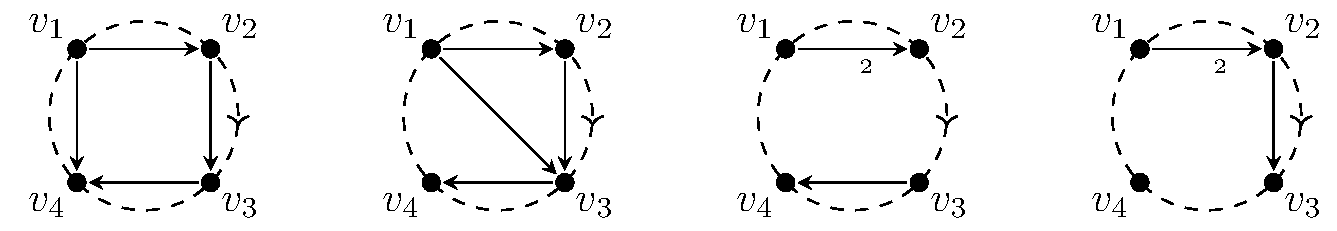}$$

Any 4-vertex acyclic quiver with $M_Q=6$ must be connected. 
A quiver from each mutation-acyclic mutation-class with $M_Q=6$ are shown below.

$$\includegraphics[alt={Three COQs, each with cyclic ordering (v1,v2,v3,v4). The leftmost has 2 arrows from v1 to v2, and just one arrow from v2 to v3, and v3 to v4. The middle COQ has one arrow from v1 to v2 and v3 to v4, and two arrows from v2 to v3. The last has 2 arrows from v1 to v2, and just one arrow from v1 to v3 and v1 to v4.}]{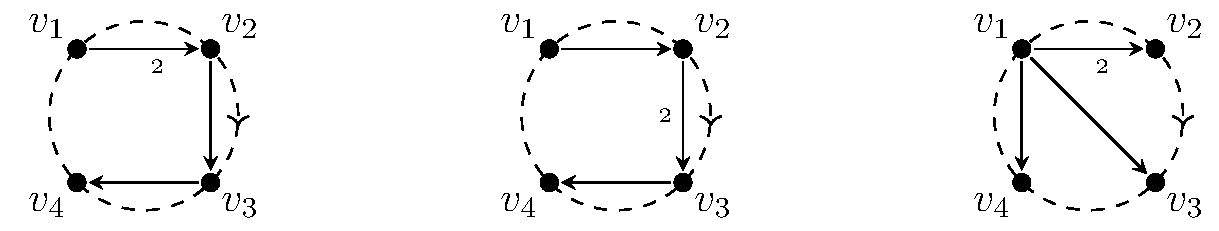}$$}

Thus the complete list of possible Alexander polynomials of mutation-acyclic quivers with~${M_Q \leq 6}$ is:

{
\renewcommand{\arraystretch}{1.2} 
\noindent
$\begin{array}{l l}
(t-1)^4 & 
 (t-1)^4 + 5t(t-1)^2 
 \\
(t-1)^4 + t(t-1)^2 &
 (t-1)^4 + 5t(t-1)^2 + t^2 
 \\
(t-1)^4 + 2t(t-1)^2 
& 
 (t-1)^4 + 5t(t-1)^2 + 4t^2 
\\
(t-1)^4 + 2t(t-1)^2 + t^2 
& 
 (t-1)^4 + 6t(t-1)^2 
\\
 (t-1)^4 + 3t(t-1)^2 
&  (t-1)^4 + 6t(t-1)^2 + t^2 
\\
 (t-1)^4 + 3t(t-1)^2 + t^2 \quad \quad 
& 
 (t-1)^4 + 6t(t-1)^2 + 4t^2 
\\
 (t-1)^4 + 4t(t-1)^2 
& \\
\end{array}
$
}

There is only one polynomial with a $4t(t-1)^2$ summand because all three mutation-classes with $M_Q=4$ have the same Alexander polynomial.
(As do two of the mutation-classes with $M_Q=5$.)
\end{example}

\hide{
\begin{figure}[ht]
{
\newcommand\cyclicVerts[5]{%
	\filldraw[black] (#4,#5)++(135:0.8cm) circle (2pt) node[above left=-1pt] {$v_1$} coordinate (v_1);
	\filldraw[black] (#4,#5)++(45:0.8cm) circle (2pt) node[above right=-1pt] {$#1$} coordinate (#1);
	\filldraw[black] (#4,#5)++(-45:0.8cm) circle (2pt) node[below right=-1pt] {$#2$} coordinate (#2);
	\filldraw[black] (#4,#5)++(-135:0.8cm) circle (2pt) node[below left=-1pt] {$#3$} coordinate (#3);
	\draw[black, dashed, decoration={markings, mark=at position 0 with {\arrow{<}}}, postaction={decorate}] (#4,#5) circle (0.8cm);
}
\includegraphics{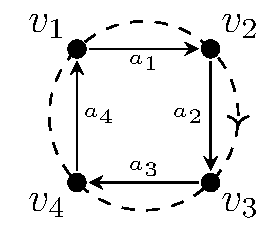}
}

\caption{A $4$-cycle quiver.}
\label{fig:4cycle}
\end{figure}
}

\begin{example}
\label{eg:extended not acyclic}
Consider the quivers associated with the elliptic root system of double extended type $E_7$ or $E_8$, denoted $E_7^{(1,1)}$ or $E_8^{(1,1)}$ and shown in Figure~\ref{fig:Extended78}. 
(Note that the quiver in Figure~\ref{fig:Extended8} is mutation equivalent to $E_8^{(1,1)}$.)
These quivers have finite mutation classes, and can be shown to be totally proper by exhaustive computation. 
(We emphasize the exhaustion; $E_8^{(1,1)}$ has $5739$ quivers, up to isomorphism, in its class.)
These quivers are not mutation-acyclic.
This has been shown by brute force, or by representation theory \cite[Section 2.3]{KellerE8}.
We sketch a new proof.

The canonical candidate cyclic orderings $\sigma_{E_7^{(1,1)}}$ and $\sigma_{E_8^{(1,1)}}$ for the quivers shown in Figure~\ref{fig:Extended8} are $(v_1, \ldots, v_9)$ and $(v_1, \ldots, v_{10})$ respectively. 
(While $(E_7^{(1,1)}, \sigma_{E_7^{(1,1)}})$ and $(E_8^{(1,1)}, \sigma_{E_8^{(1,1)}})$ are totally proper, we do not rely on this fact.)
So we treat $E_7^{(1,1)}$ and $E_8^{(1,1)}$ as COQs with their respective cyclic ordering.
Their Markov invariants are:
$$M_{E_7^{(1,1)}} = 10 + 4 - 6 = 8,\quad M_{E_8^{(1,1)}} = 11 + 4 - 6 = 9.$$
By Corollary~\ref{cor:markov finite}, the only acyclic quivers that $E_7^{(1,1)}$ (resp., $E_8^{(1,1)}$) could be mutation equivalent to are tree quivers (with no parallel arrows).

Up to isomoprhism, there are $47$ (unoriented) trees with $9$ vertices and $106$ trees with $10$ vertices. 
All orientations of each of these trees are mutation equivalent.
By Theorem~\ref{th:wiggle/winding}, all cyclic orderings of a tree are wiggle equivalent.
Thus any quiver mutation equivalent to a tree on $9$ (resp., $10$) vertices must have one of the $47$ (resp., $106$) Alexander polynomials. (Actually the list is slightly smaller, as some non-isomorphic trees have the same Alexander polynomial.)

So it suffices to check if 
\begin{align*}
\Delta_{E_7^{(1,1)}}(t) &= t^9 - t^8 - 2t^5 + 2 t^4 + t - 1, \text{ and} \\
\Delta_{E_8^{(1,1)}}(t) &= t^{10} - t^9 + t^7 - t^6 - t^4 + t^3 - t + 1 
\end{align*}
appear in these lists.
A (still tedious, but straightforward and human verifiable) computation shows that they are not.
The computation can be further simplified using the work of Amanda Schwartz \cite{SchwartzTree}.
\end{example}

\begin{figure}[ht]
{
\newcommand\spacingh{1.5cm}
\newcommand\hgtOne{1cm}
\newcommand\hgtTwo{-1cm}
\includegraphics[alt={A nine vertex quiver with vertices v1, ..., v9 and arrows as follows: 2 arrows from v5 to v1, one arrow from v1 to v2, v2 to v5, v1 to v3, v3 to v5, v1 to v4, v4 to v5, v3 to v6, v6 to v7, v4 to v8, and v8 to v9. Thus there are oriented three cycles on v5, v1, and any one of v2, v3, or v4.}]{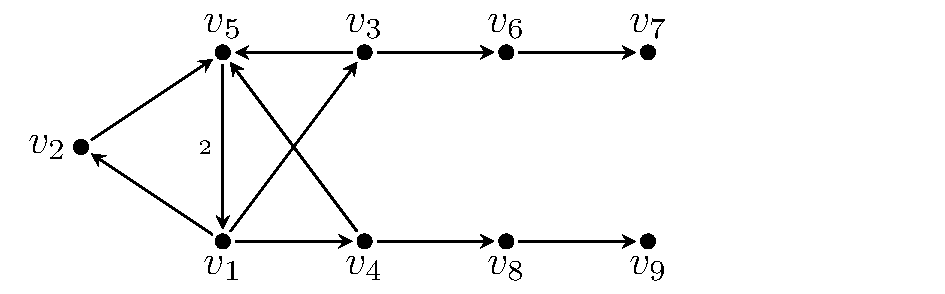}

\includegraphics[alt={A ten vertex quiver with vertices v1, ..., v10 and arrows as follows: 2 arrows from v5 to v1, one arrow from v1 to v2, v2 to v5, v1 to v3, v3 to v5, v1 to v4, v4 to v5, v3 to v6, v6 to v7, v7 to v8, v8 to v9, and v4 to v10. Thus there are oriented three cycles on v5, v1, and any one of v2, v3, or v4.}]{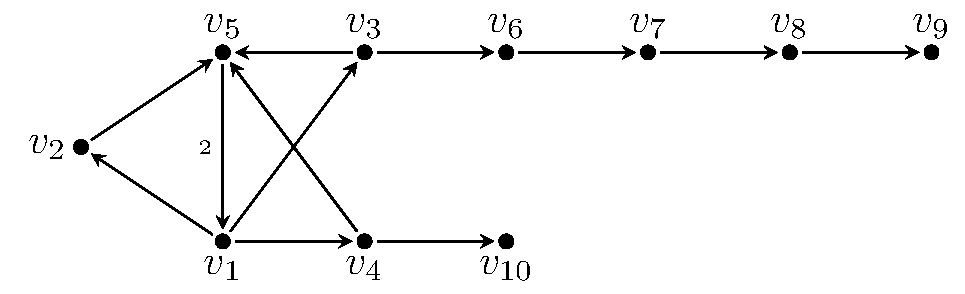}
}

\caption{The extended affine quivers $E_7^{(1,1)}$ and $E_8^{(1,1)}$.}
\label{fig:Extended78}
\end{figure}

\begin{example}
\label{eg:glue3vert}
Choose positive integers $a,b,c$.
Let $Q$ be the $4$-vertex COQ below:
{
\newcommand\cyclicVerts[5]{%
	\filldraw[black] (#4,#5)++(135:0.8cm) circle (2pt) node[above left=-1pt] {$v_1$} coordinate (v_1);
	\filldraw[black] (#4,#5)++(45:0.8cm) circle (2pt) node[above right=-1pt] {$#1$} coordinate (#1);
	\filldraw[black] (#4,#5)++(-45:0.8cm) circle (2pt) node[below right=-1pt] {$#2$} coordinate (#2);
	\filldraw[black] (#4,#5)++(-135:0.8cm) circle (2pt) node[below left=-1pt] {$#3$} coordinate (#3);
	\draw[black, dashed, decoration={markings, mark=at position 0 with {\arrow{<}}}, postaction={decorate}] (#4,#5) circle (0.8cm);
}
$$\includegraphics[alt={A four vertex COQ with cyclic ordering (v1, v2, v3, v4) and arrows as follows: a arrows from v1 to v2; b arrows from v2 to v3 and from v2 to v4; c arrows from v3 to v1 and v4 to v1.}]{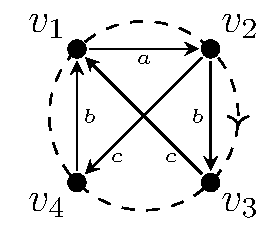}$$}
\vspace{-15pt}

Then $M_Q = a^2 + 2 b^2 + 2c^2 - 2 abc.$
Let $Q'$ be the subquiver supported by $v_1, v_2,v_3$, thus $M_{Q'} = a^2 + b^2 + c^2 - abc.$
Note that $M_Q = 2 M_{Q'} - a^2.$
Thus $M_Q < 0$ whenever~$M_{Q'} < \frac{a^2}{2}$.
(This last inequality is generally satisfied for mutation-infinite quivers; as $M_{Q'}$ is constant on the mutation class, but the number of arrows grows arbitrarily large, we can find many quivers with some weight $a$ much larger than $M_{Q'}$.)
Thus by Corollary~\ref{cor:acyclic positive}, such $Q$ cannot be mutation-acyclic.
\end{example}

\hide{ 
\begin{remark}
The COQ $Q$ constructed in Example~\ref{eg:glue3vert} can be constructed in another way. 
Start with any $3$-vertex quiver $Q'$; create a copy with all arrows reversed; pick an arrow $v \rightarrow w$ in $Q'$ and identify $v$ in $Q'$ with $w$ in the copy, and likewise $w$ in $Q'$ with $v$ in the copy.
In general, whenever we have two quivers with a common subquiver we can identify that common subquiver to create a larger quiver.
The Markov invariant of the resulting quiver will contain all the terms from both their respective Markov invariants.
Generally giving fairly simple inequalities for when the result must not be mutation-acyclic.
\end{remark}
} 

\begin{example}
\label{eg:glueeven}
We next illustrate a way of combining Corollary~\ref{cor:markov finite} with another mutation invariant \cite{Seven3x3}, the multiset of GCDs of all the integers in the columns (or rows) of the $B$-matrix. 
We construct a COQ as in Example~\ref{eg:glue3vert} with $a=b=2$ and $c=4$.
Thus $M_{Q} = 12 \geq 3$, and we cannot use Corollary~\ref{cor:acyclic positive} to conclude that $Q$ is not mutation-acyclic.
However, note that the GCD of the entries of $B_Q$ is $2$.
Thus if $Q$ is mutation-acyclic, it must be mutation equivalent to a connected acyclic quiver with the same GCD and $M_Q=12$.
There are two such quivers, both trees with $2$ arrows between each pair of neighboring vertecies.
The $B$-matrices of these acyclic quivers have determinant $0$ and $16$, but $\det(B_Q)=144$. 
Thus $Q$ is not mutation-acyclic.
\end{example}

It is natural to ask when the Alexander polynomial is unimodal (say up to the signs of the coefficients).
If we make no assumptions on the quiver, then the general answer seems complicated (note that the Alexander polynomial of the quivers in Example~\ref{eg:extended not acyclic} are not).
We give two examples from acyclic quivers. 

\begin{example}
\label{eg:not unimodal}
Let $Q$ be the COQ shown below:
{
\newcommand\cyclicVerts[5]{%
	\filldraw[black] (#4,#5)++(135:0.8cm) circle (2pt) node[above left=-1pt] {$v_1$} coordinate (v_1);
	\filldraw[black] (#4,#5)++(45:0.8cm) circle (2pt) node[above right=-1pt] {$#1$} coordinate (#1);
	\filldraw[black] (#4,#5)++(-45:0.8cm) circle (2pt) node[below right=-1pt] {$#2$} coordinate (#2);
	\filldraw[black] (#4,#5)++(-135:0.8cm) circle (2pt) node[below left=-1pt] {$#3$} coordinate (#3);
	\draw[black, dashed, decoration={markings, mark=at position 0 with {\arrow{<}}}, postaction={decorate}] (#4,#5) circle (0.8cm);
}
$$\includegraphics[alt={A COQ with cyclic ordering (v1,v2,v3,v4) and 2 arrows each between v1 and v2, v2 and v3, and v1 and v3, as well as 4 arrows from v3 to v4.}]{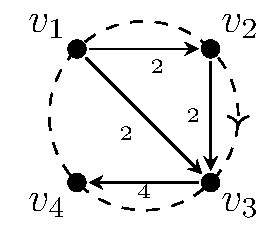}$$}

Then $\Delta_Q(t) = t^4 + 32 t^3 - 2 t^2 + 32t + 1$ is not unimodal.
\end{example}

\begin{example}
\label{eg:bipartite acyclic}
Consider the COQ $Q$ below:
{
\newcommand\cyclicVerts[5]{%
	\draw[gray, dashed, decoration={markings, mark=at position 0 with {\arrow{<}}}, postaction={decorate}] (#4,#5) circle (0.8cm);
	\filldraw[black] (#4,#5)++(135:0.8cm) circle (2pt) node[above left=-1pt] {$v_1$} coordinate (v_1);
	\filldraw[black] (#4,#5)++(45:0.8cm) circle (2pt) node[above right=-1pt] {$#1$} coordinate (#1);
	\filldraw[black] (#4,#5)++(-45:0.8cm) circle (2pt) node[below right=-1pt] {$#2$} coordinate (#2);
	\filldraw[black] (#4,#5)++(-135:0.8cm) circle (2pt) node[below left=-1pt] {$#3$} coordinate (#3);
}
$$\includegraphics[alt={A COQ with cyclic ordering (v1,v2,v3,v4). There are a1 arrows from v1 to v3, a2 from v2 to v3, a3 from v1 to v4, and a4 from v2 to v4.}]{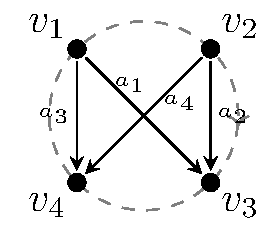}$$}
We compute the Alexander polynomial 
\begin{align*}
\Delta_Q(t) =& t^4 + (a_1^2 + a_2^2 + a_3^2 + a_4^2-4) t^3 \\
&+ (a_2^2a_3^2 - 2a_1a_2a_3a_4 + a_1^2a_4^2 - 2a_1^2 - 2a_2^2 - 2a_3^2 - 2a_4^2 + 6) t^2 \\
&+ (a_1^2 + a_2^2 + a_3^2 + a_4^2-4) t + 1 \\
=& (t-1)^4 + (a_1^2 + a_2^2 + a_3^2 + a_4^2) t (t-1)^2 + (a_2 a_3-a_1a_4)^2 t^2.\\
\end{align*}
A polynomial is \emph{$\gamma$-positive} if it can be expressed as $\sum_i \gamma_i t^i (t+1)^{n-2i}$ for some nonnegative coefficients $\gamma_i$ \cite{AGammaPositivity}.
Our second expression above is almost of this form, except that we have $(t-1)$ instead of $(t+1)$. 
It is not hard to show that $\Delta_Q(t)$ will be unimodal (indeed, $\gamma$-positive) when $64 < 4(a_1^2 + a_2^2 + a_3^2 + a_4^2) < (a_2 a_3-a_1a_4)^2$.
\end{example}

Finally, we note that Lemmas~\ref{lem:mutates like B v1gen}~\ref{lem:mutates like B v2gen} extend to the admissible quasi-Cartan companion $A_U= U + U^T$ (discussed in \cite{COQ}[Remark~16.2]). 

\begin{corollary}
\label{cor:A mutates like B}
Fix a totally proper COQ $Q$ and a proper vertex $v_j$.
Let $U_Q$ be a unipotent companion with linear order $<$ so that $v_i < v_j$ (resp., $v_j < v_i$) for all $v_i \in \In(v_j)$ (resp., $\Out(v_j)$).
Then $M_Q(v_j, \epsilon) (U_Q + U_Q^T) M_Q(v_j, \epsilon)$ 
is an admissible quasi-Cartan companion of $\mu_{v_j}(Q)$ when $\epsilon = -1$ (resp., $1$).
\end{corollary}

\end{document}